\documentclass[10pt]{amsart}
\usepackage{tikz-cd}
\usetikzlibrary{arrows.meta}
\usepackage{amssymb}
\usepackage{mathrsfs}
\usepackage[shortlabels]{enumitem}
\usepackage[colorlinks,linkcolor=black,citecolor=black,urlcolor=black]{hyperref}
\usepackage[color = blue!20, bordercolor = black, textsize = tiny]{todonotes}
\usepackage[capitalise]{cleveref}
\usepackage{comment}

\numberwithin{equation}{section}
\newtheorem{thm}[subsection]{Theorem}
\newtheorem{cor}[subsection]{Corollary}
\newtheorem{lem}[subsection]{Lemma}
\newtheorem{prop}[subsection]{Proposition}

\theoremstyle{definition}

\newtheorem{rmk}[subsection]{Remark}

\newtheorem{const}[subsection]{Construction}

\newcommand{\bF}{\mathbf{F}}

\newcommand{\A}{\mathbb{A}}

\newcommand{\N}{\mathbb{N}}
\renewcommand{\P}{\mathbb{P}}
\newcommand{\Q}{\mathbb{Q}}
\newcommand{\R}{\mathbb{R}}

\newcommand{\T}{\mathbb{T}}

\newcommand{\Z}{\mathbb{Z}}

\newcommand{\cC}{\mathcal{C}}

\newcommand{\cF}{\mathcal{F}}
\newcommand{\cG}{\mathcal{G}}

\newcommand{\cM}{\mathcal{M}}

\newcommand{\rC}{\mathrm{C}}
\newcommand{\rD}{\mathrm{D}}

\newcommand{\rM}{\mathrm{M}}

\DeclareMathOperator{\Hom}{Hom}
\DeclareMathOperator{\Spec}{Spec}

\newcommand{\cok}{\mathrm{cok}}

\newcommand{\colim}{\mathop{\mathrm{colim}}}

\newcommand{\id}{\mathrm{id}}
\newcommand{\ul}{\underline}
\newcommand{\ol}{\overline}

\newcommand{\pt}{\mathrm{pt}}
\newcommand{\lSm}{\mathrm{lSm}}

\newcommand{\lSch}{\mathrm{lSch}}

\newcommand{\gp}{\mathrm{gp}}

\newcommand{\sing}{\mathrm{sing}}

\DeclareMathOperator{\fib}{fib}

\newcommand{\Sh}{\mathrm{Sh}}

\newcommand{\Cat}{\mathrm{Cat}}
\newcommand{\Mod}{\mathrm{Mod}}

\newcommand{\ad}{\mathrm{ad}}

\newcommand{\CAlg}{\mathrm{CAlg}}

\newcommand{\op}{\mathrm{op}}

\newcommand{\eff}{\mathrm{eff}}

\newcommand{\SH}{\mathrm{SH}}

\newcommand{\ex}{\mathrm{ex}}
\newcommand{\DA}{\mathrm{DA}}

\newcommand{\DM}{\mathrm{DM}}

\newcommand{\Zar}{\mathrm{Zar}}

\newcommand{\letale}{\mathrm{l\acute{e}t}}
\newcommand{\ket}{\mathrm{k\acute{e}t}}
\newcommand{\setale}{\mathrm{s\acute{e}t}}
\newcommand{\mot}{\mathrm{mot}}

\newcommand{\et}{\mathrm{\acute{e}t}}

\newcommand{\Th}{\mathrm{Th}}

\newcommand{\add}{\mathrm{add}}

\usepackage[bbgreekl]{mathbbol}
\usepackage{relsize}

\DeclareSymbolFontAlphabet{\mathbb}{AMSb} 
\DeclareSymbolFontAlphabet{\mathbbl}{bbold}

\begin{document}
\title[Six-functor formalism for Kummer \'etale cohomology]{Six-functor formalism for Kummer \'etale cohomology of log schemes}
\author{Doosung Park}
\address{Department of Mathematics and Informatics, University of Wuppertal, Germany}
\email{dpark@uni-wuppertal.de}
\date{\today}
\begin{abstract}
We establish a Grothendieck six-functor formalism for Kummer \'etale cohomology including Poincar\'e duality for every separated vertical exact log smooth morphism of noetherian fs log schemes $f\colon X\rightarrow S$ when the coefficient ring $\Lambda$ is killed by an integer invertible on $S$. This is done via log \'etale rigidity \[\mathrm{D}_{\mathrm{l\acute{e}t}}(S,\Lambda)\simeq \mathrm{DA}_{\mathrm{l\acute{e}t}}(S,\Lambda).\] To achieve this, we also prove that Kummer \'etale cohomology satisfies $\mathbb{A}^1$-invariance, invariance under virtual isomorphisms, log cdh-descent, and invariance under verticalization.
\end{abstract}
\maketitle

\section{Introduction}

Fujiwara, Kato, and Nakayama \cite{Nak1}, \cite{Ill}, \cite{Nak2} studied the Grothendieck six-functor formalism for Kummer \'etale and log \'etale derived categories.
These results include the proper base change theorem, smooth base change theorem, K\"unneth formula, invariance under log blow-ups, etc.
However,
Poincar\'e duality for Kummer \'etale cohomology has not been established in the following generality:
\[
f^!
\simeq
f^*(d)[2d]
\colon
\rD_\ket(S,\Lambda) \to \rD_\ket(X,\Lambda)
\]
for every separated vertical exact log smooth morphism $f\colon X\to S$ of finite dimensional noetherian fs log schemes with pure relative dimension $d$,
where $\Lambda$ is a torsion commutative ring invertible in $S$.
Here, $\rD_\ket(S,\Lambda)$ denotes the derived $\infty$-category of Kummer \'etale sheaves of $\Lambda$-modules on the small Kummer \'etale site $S_\ket$.
See \cite[Theorem 7.5]{Nak1} and \cite[Proposition 4.4]{MR1622751} for known cases,
and see also \cite[Remark 14.2.1]{Nak2}.
Nakayama \cite[Theorem 14.2(2), Remark 14.2.2]{Nak2} also proved a weaker version
\[
f_* \ul{\Hom}(\cF,\Lambda)(d)[2d]
\simeq
\ul{\Hom}(f_*\cF,\Lambda)
\]
for every locally constant and constructible sheaf $\cF$ of $\Lambda$-modules on $X_\ket$ assuming that $f$ is also proper,
where $\ul{\Hom}$ is internal Hom.

In this paper,
we establish the following version of the Grothendieck six-functor formalism including Poincar\'e duality for $\rD_\ket(-,\Lambda)$:

\begin{thm}
\label{intro.1}
Let $B$ be a finite dimensional noetherian base scheme,
and let $\Lambda$ be a torsion commutative ring killed by an integer invertible on $B$.
Then in the category $\lSch_B$ of fs log schemes of finite type over $B$,
the following properties hold for $\rD_\ket(-,\Lambda)$:
\begin{enumerate}
\item[\textup{(1)}]
For every morphism $f\colon X\to S$ in $\lSch_B$,
the functor
\[
f^*\colon \rD_\ket(S,\Lambda)\to \rD_\ket(X,\Lambda)
\]
admits a right adjoint $f_*$.
If $f$ is exact log smooth,
then $f^*$ admits a left adjoint $f_\sharp$.
\item[\textup{(2)}] \emph{$\A^1$-invariance.}
Consider the projection $p\colon X\times \A^1\to X$ with $X\in \lSch_B$.
Then $p^*$ is fully faithful.
\item[\textup{(3)}]
\emph{Invariance under verticalization}.
Let $f\colon X\to S$ be an exact log smooth morphism in $\lSch_B$,
and let $j\colon X-\partial_S X\to X$ be the obvious open immersion,
where $X-\partial_S X$ is the largest open subscheme of $X$ that is vertical over $S$.
Then the natural transformation
\[
f^*
\to
j_*j^* f^*
\]
is an isomorphism.
\item[\textup{(4)}]
\emph{Invariance under virtual isomorphisms.} Let $f\colon X \to S$ be a virtual isomorphism in $\lSch_B$
in the sense that  $\ul{f}$ is an isomorphism and $\ol{\cM}_{S}^\gp \simeq \ol{\cM}_{X}^\gp$.
Then $f^*$ is an equivalence of $\infty$-categories.
\item[\textup{(5)}]
\emph{Localization property.}
Let $i\colon Z\to S$ be a strict closed immersion in $\lSch_B$ with the open complement $j$.
Then the pair of functors $(i^*,j^*)$ is jointly conservative,
and $i_*$ is fully faithful.
\item[\textup{(6)}]
Let $(\lSch_B)_\mathrm{sep}$ be the subcategory of $\lSch_B$ spanned by separated morphisms.
Then there exists a functor
\[
(-)_!\colon (\lSch_B)_\mathrm{sep}\to
\Cat_\infty
\]
such that $f_!\simeq f_*$ if a morphism $f$ in $\lSch_B$ is proper and $f_!\simeq f_\sharp$ if $f$ is an open immersion.
Furthermore,
for every separated morphism $f$ in $\lSch_B$,
$f_!$ admits a right adjoint $f^!$.
\item[\textup{(7)}] \emph{Projection formula}. Let $f\colon X\to S$ be a separated morphism in $\lSch_B$.
Then there exists a natural isomorphism
\[
f_!\cF\otimes \cG
\simeq
f_!(\cF\otimes f^*\cG)
\]
for $\cF\in \rD_\ket(X)$ and $\cG\in \rD_\ket(S)$.
\item[\textup{(8)}] \emph{Kato--Nakayama's base change property \cite[Theorem 5.1]{Nak1}, \cite[Theorem 12.7]{Nak2}}. 
Let
\[
\begin{tikzcd}
X'\ar[d,"f'"']\ar[r,"g'"]&
X\ar[d,"f"]
\\
S'\ar[r,"g"]&
S
\end{tikzcd}
\]
be a cartesian square in $\lSch_B$ such that either $f$ or $g$ is exact.
Assume one of the following conditions:
\begin{enumerate}
\item[\textup{(i)}] $f$ is proper.
\item[\textup{(ii)}] $g$ is exact log smooth.
\end{enumerate}
Then
\[
g^*f_*
\simeq
f_*'g'^*.
\]
\item[\textup{(9)}] \emph{Poincar\'e duality}. Let $f\colon X\to S$ be a separated vertical exact log smooth morphism in $\lSch_B$ with pure relative dimension $d$.
Then
\[
f^! 
\simeq
f^*(d)[2d].
\]
\item[\textup{(10)}] \emph{A special case of Lefschetz duality}.
Let $f\colon X\to S$ be a separated log smooth morphism in $\lSch_B$ with pure relative dimension $d$,
and let $j\colon X-\partial_S X\to X$ be the obvious open immersion.
Assume that $S$ has the trivial log structure.
Then
\[
f^!\simeq j_\sharp j^*f^*(d)[2d].
\]
\end{enumerate}
\end{thm}
See \S \ref{proof} for the proof.
Here is a summary of the comparison of this theorem with the literature:
\begin{itemize}
\item In (1), the existence of $f_*$ follows formally from known results.
The existence of $f_\sharp$ is new.
\item (2) is new.
\item (3) is new. In the absolute case, this is due to Fujiwara--Kato  \cite[Theorem 7.4]{Ill}.
\item (4) is new.
\item (5) is a direct consequence of \cite[2.8.3]{Nak1}.
\item (6) is an $\infty$-categorical enrichment of \cite[5.4]{Nak1}.
\item (7) is new.
\item (8) is an immediate generalization of \cite[Theorem 5.1]{Nak1} and \cite[Theorem 12.7]{Nak2} (with the simplified condition in \cite[Remark 12.7.1]{Nak2}) to the unbounded case.
\item (9) is new and improves \cite[Theorem 7.5]{Nak1}, \cite[Proposition 4.4]{MR1622751}, and \cite[Theorem 14.2(2)]{Nak2} for the vertical noetherian case. 
\item (10) is new and improves a special case of \cite[Theorem 7.5]{Nak1}.
\end{itemize}

For the case of \'etale cohomology of schemes,
there are two different proofs of Poincar\'e duality.
The first proof is of course the original one in SGA4 \cite[Th\'eor\`eme 3.2.5 in Expos\'e XVIII]{SGA43}.
The second proof proceeds via Ayoub's relative \'etale rigidity \cite[Th\'eor\`eme 4.1]{Ayo14},
which asserts
\[
\rD_\et(S,\Lambda)
\simeq
\DA_\et(S,\Lambda)
\]
for every noetherian scheme $S$ and a commutative ring $\Lambda$ killed by an integer invertible on $S$.
The notable differences between Ayoub’s theorem and Suslin–-Voevodsky’s étale rigidity \cite[Theorems 7.20, 9.35]{MVW} are that Ayoub’s theorem applies over a general base $S$, is formulated in a $\P^1$-stable setting, and does not use transfers.
Combining this with Ayoub's motivic six-functor formalism for $\DA_\et(S,\Lambda)$ \cite[\S 4.5]{Ayo072} (also \cite[Theorem 2.4.50]{CD19} to avoid the quasi-projectivity assumption),
we see that the same six-functor formalism holds for $\rD_\et(S,\Lambda)$.
As explained in \cite[Remarque 4.16]{Ayo14},
this reproves Poincar\'e duality for \'etale cohomology.

We prove Theorem \ref{intro.1} using relative log \'etale rigidity.
We have a motivic homotopy category $\SH(S)$ in \cite[Definition 2.29]{logA1} extending the original $\SH$ of Morel and Voevodsky \cite{MV} from schemes to fs log schemes $S$.
We also have the $\infty$-category $\DA(S,\Lambda)$ of $\Lambda$-modules in $\SH(S)$ and its log \'etale hyperlocalization $\DA_\letale(S,\Lambda)$.
Let $\DA_\letale^\ex(S,\Lambda)$ be the full subcategory generated under colimits, shifts, and Tate twists by exact log smooth motives.
For the underlying scheme $\ul{S}$, \cite[Proposition 2.31]{logA1} implies
\[
\DA_\et(\ul{S},\Lambda)
\simeq
\DA_\letale^\ex(\ul{S},\Lambda)
\simeq
\DA_\letale(\ul{S},\Lambda).
\]
However, if $S$ has a nontrivial log structure, then $\DA_\et(\ul{S},\Lambda)$ is not equivalent to $\DA_\letale^\ex(S,\Lambda)$ and $\DA_\letale(S,\Lambda)$.
Our main result on the relative log \'etale rigidity is as follows:

\begin{thm}[Theorems \ref{rig.4} and \ref{rig.6}]
\label{intro.2}
Let $S$ be a finite dimensional noetherian fs log scheme,
and let $\Lambda$ be a torsion commutative ring killed by an integer invertible on $S$.
Then there exist equivalences of symmetric monoidal $\infty$-categories
\[
\rD_\ket(S,\Lambda)\simeq
\DA_\letale^\ex(S,\Lambda),
\;
\rD_\letale(S,\Lambda)\simeq
\DA_\letale(S,\Lambda).
\]
\end{thm}

To prove this,
we need to show that log \'etale cohomology is $\A^1$-invariant in \S \ref{5} and invariant under verticalization in \S \ref{8},
which are not covered in \cite{Nak1}, \cite{Ill}, and \cite{Nak2}.
The invariance under verticalization is the main technical part of this paper.

We know that $\DA_\letale^\ex(S,\Lambda)$ enjoys the six-functor formalism in \cite[Theorems 1.2.1, 1.3.1]{logsix} for finite dimensional noetherian fs log schemes admitting charts Zariski locally.
In \S \ref{proof},
we explain how to remove the assumption of admitting charts Zariski locally.

\subsection*{Organization of the paper}
In \S \ref{unbounded}, we explain how to promote results for bounded complexes to unbounded complexes under the noetherian assumption on fs log schemes.
In \S \ref{motive}, we recall the definitions of $\DA_\letale(S,\Lambda)$ and $\DA_\letale^\ex(S,\Lambda)$,
and we show that the pushforward functor for any finite type morphism preserves colimits.
In \S \ref{4}, we prove Poincar\'e duality for the log \'etale case.
In \S \ref{5}, we prove that log \'etale cohomology is $\A^1$-invariant and $\square$-invariant,
where $\square:=(\P^1,\infty)$.
In \S \ref{6}, we prove that log \'etale cohomology is invariant under virtual isomorphisms.
In \S \ref{7}, we discuss log v-descent and log cdh-descent for log \'etale cohomology.
In \S \ref{8}, we prove that log \'etale cohomology is invariant under verticalization.
In \S \ref{rig}, combining the results from earlier sections and imitating the proofs of Ayoub's relative \'etale rigidity \cite[Corollaire 4.11]{Ayo14} for the fully faithful part and Cisinski--D\'eglise's proof of relative \'etale rigidity \cite[Theorem 4.5.2]{CDetale} for theh essentially surjective part, we establish relative log \'etale rigidity.
In \S \ref{proof}, we remove the assumption of admitting charts Zariski locally used in \cite{logsix} for $\DA_\letale^\ex(-,\Lambda)$.

\subsection*{Convention} Our standard reference for the notation and
terminology in log geometry is Ogus’s
book \cite{Ogu}.
An fs log scheme $S$ is called quasi-compact, quasi-separated, separated, noetherian, or finite dimensional if the same holds for the underlying scheme $\ul{S}$.
A morphism of fs log schemes $f$ is called affine, separated, proper, finite type, or finite if the same holds for the underlying morphism of schemes $\ul{f}$.
For a stable $\infty$-category $\cC$ with objects $X$ and $Y$,
let $\hom_\cC(X,Y)$ denote the hom spectrum.

\subsection*{Acknowledgements}

This research was conducted in the framework of the DFG-funded research training group GRK 2240: \emph{Algebro-Geometric Methods in Algebra, Arithmetic and Topology}.

\subsection*{AI declaration}

The author used GPT-5.6 during the preparation of this manuscript to check grammar and spelling and to assist in reviewing the mathematical proofs.

\section{From bounded complexes to unbounded complexes}
\label{unbounded}

For a qcqs fs log scheme $S$,
let $S_\ket$ (resp.\ $S_\letale$) denote the small Kummer (resp.\ log) \'etale site of $S$, see \cite[2.1, 2.5]{Ill} (resp.\ \cite[3.3, 3.4, Proposition 3.5]{Nak2}).
The underlying category is the category of qcqs fs log schemes Kummer (resp.\ log) \'etale over $S$,
Every morphism in $S_\ket$ (resp.\ $S_\letale$) is automatically Kummer (resp.\ log) \'etale.
Moreover, $S_\ket$ (resp.\ $S_\letale$) has finite colimits.

A family of morphisms $\{U_i\to X\}_{i\in I}$ of fs log schemes is called a \emph{Kummer} (resp.\ \emph{log}) \emph{\'etale covering} if each $U_i\to X$ is Kummer (resp.\ log) \'etale and $\amalg_{i\in I}U_i\to X$ is surjective (resp.\ universally surjective).

For a qcqs fs log scheme $S$ and $X\in S_\ket$ (resp.\ $X\in S_\letale$), let $\Lambda_X\in \rD_\ket(S,\Lambda)$ (resp.\ $\rD_\ket(S,\Lambda)$) be the object represented by $X$.

We refer to \cite[D\'efinition 2.3 in Expos\'e VI]{SGA42} for the notion of coherent topos.

\begin{prop}
\label{dim.11}
Let $S$ be a qcqs fs log scheme.
Then $\Sh(S_\ket)$ and $\Sh(S_\letale)$ are coherent topoi.
\end{prop}
\begin{proof}
Recall that $S_\ket$ and $S_\letale$ have finite colimits.
To conclude,
observe that every object in $S_\ket$ and $S_\letale$ is qcqs in the sense of topos by \cite[Lemma 3.14(1)]{Nak2}.
\end{proof}

For a qcqs fs log scheme $S$,
we say that $S$ has \emph{finite Kummer} (resp.\ \emph{log}) \emph{\'etale cohomological dimension} if there exists an integer $N$ such that
\[
H_\ket^i(S,\cF)=0\text{ (resp.\ } H_\letale^i(S,\cF)=0)
\]
for every torsion sheaf of abelian groups on $S_\ket$ (resp.\ $S_\letale$) and integer $i>N$.

\begin{prop}
\label{dim.1}
Let $S$ be a strictly local noetherian scheme,
and let $X$ be an fs log scheme of finite type over $S$.
Then $X$ has finite Kummer \'etale and log \'etale cohomological dimensions.
\end{prop}
\begin{proof}
We know that $\ul{X}$ has finite \'etale cohomological dimension by \cite[Theorem 1.1.5]{CDetale}.
The conclusion follows from this result and \cite[Theorem 7.2]{Nak2}.
\end{proof}

\begin{prop}
\label{dim.6}
Let $S$ be a qcqs fs log scheme having finite Kummer \'etale cohomological dimension,
and let $\Lambda$ be a torsion commutative ring.
Then the functor $R\Gamma_\ket(S,-)$ on $\rD_\ket(S,\Lambda)$ preserves colimits.
The same holds in the log \'etale case too.
\end{prop}
\begin{proof}
By \cite[Lemma 1.1.7]{CDetale},
$R\Gamma_\ket(S,-)$ preserves filtered colimits.
Since it preserves finite limits and hence finite colimits,
it preserves colimits.
The same argument applies to the log \'etale case.
\end{proof}

For an fs log scheme $S$,
a \emph{geometric point $\ol{s}$ of $S$} is defined to be a strict morphism $\ol{s} \to S$ such that the underlying scheme of $\ol{s}$ is a geometric point of $\ul{S}$.

\begin{prop}
\label{dim.2}
Let $\Lambda$ be a torsion commutative ring,
let
\[
\begin{tikzcd}
X'\ar[d,"f'"']\ar[r,"g'"]&
X\ar[d,"f"]
\\
S'\ar[r,"g"]&
S
\end{tikzcd}
\]
be a cartesian square of noetherian fs log schemes,
where $g$ exhibits $S'$ as a strict henselization of $S$ at a geometric point $\ol{s}$.
Assume that $f$ is of finite type.
\begin{itemize}
\item[\textup{(1)}]
The functor
\[
g'^*\colon \rC(\Sh(X_\ket),\Lambda)\to \rC(\Sh(X_\ket'),\Lambda)
\]
preserves the fibrations of the projective model structures,
and we have
\[
\colim_{U} R\Gamma_\ket(U,\cF\vert_U)
\simeq
R\Gamma_\ket(X',\cF\vert_{X'})
\]
for every $\cF\in \rD_\ket(X,\Lambda)$,
where the colimit runs over the strict \'etale neighborhoods of $\ol{s}$.
\item[\textup{(2)}]
We have
\[
g^*f_*
\simeq
f_*'g'^*
\colon
\rD_\ket(X,\Lambda)\to \rD_\ket(S',\Lambda).
\]
\end{itemize}
The analogous statements hold for the log \'etale case too.
\end{prop}
\begin{proof}
(1) There exists a projective system $\{S_i\}_{i\in I}$ with strict affine \'etale transition maps such that $S'\simeq \lim_{i\in I} S_i$.
By \cite[Corollary 3.5.7]{AHLS},
we have
\[
X_\ket'
\simeq
\lim_i (X_i)_{\ket},
\;
X_\letale'
\simeq
\lim_i (X_i)_{\letale},
\]
where $X_i:=X\times_S S_i$.
Hence we only need to check the conditions (i)--(iv) in \cite[Lemma 1.1.12]{CDetale}.
The condition (i) follows from Proposition \ref{dim.11},
the conditions (ii) and (iii) are obvious,
and the condition (iv) follows from Proposition \ref{dim.1}.

(2) Argue as in \cite[Lemme 4.2]{Ayo14} or \cite[Theorem 1.1.14]{CDetale}, but use part (1) as well.
\end{proof}

Using the log blow-up invariance of Kummer \'etale cohomology due to Fujiwara and Kato \cite[Theorem 6.2]{Ill},
Nakayama \cite[Theorem 9.4(b)]{Ill}, \cite[Proposition 5.4(2)]{Nak2} proved that the functor
\[
\kappa^*\colon \rD_\ket^+(S,\Lambda)\to \rD_\letale^+(S,\Lambda)
\]
induced by the inclusion $S_\ket\to S_\letale$ is fully faithful for every fs log scheme $S$ and torsion commutative ring $\Lambda$.
If we assume that $S$ is noetherian, then we can extend this result to the unbounded case as follows.

\begin{prop}
\label{dim.5}
Let $S$ be a noetherian fs log scheme,
and let $\Lambda$ be a torsion commutative ring.
Then the functor
\[
\kappa^*\colon \rD_\ket(S,\Lambda)\to \rD_\letale(S,\Lambda)
\]
is fully faithful.
\end{prop}
\begin{proof}
We need to show
\[
\cF
\xrightarrow{\simeq}
\kappa_* \kappa^* \cF
\]
for every $\cF\in \rD_\ket(S,\Lambda)$.
Consider a strict henselization $g\colon S'\to S$ at a geometric point $\ol{s}$ of $S$.
It suffices to show
\begin{equation}
\label{dim.5.1}
g^*\cF
\xrightarrow{\simeq}
g^*\kappa_*\kappa^* \cF.
\end{equation}

As in \cite[Proof of Theorem 1.1.14]{CDetale},
the functors of chain complexes of sheaves of $\Lambda$-modules
\[
g^*\colon \rC_\ket(S,\Lambda)
\to
\rC_\ket(S',\Lambda),
\;
g^*\colon\rC_\letale(S,\Lambda)
\to
\rC_\letale(S',\Lambda)
\]
preserve finite limits, weak equivalences, and fibrations of the projective model structures.
Moreover,
the functors
\[
\kappa_*\colon \rC_\letale(S,\Lambda)
\to
\rC_\ket(S,\Lambda),
\;
\kappa_*\colon \rC_\letale(S',\Lambda)
\to
\rC_\ket(S',\Lambda)
\]
are right Quillen functors for the projective model structures.
Hence
\begin{equation}
\label{dim.5.2}
R(g^*\kappa_*)
\simeq
Rg^* R\kappa_*
\simeq
g^* R\kappa_*,
\;
R(\kappa_* g^*)
\simeq
R\kappa_* Rg^*
\simeq
R\kappa_* g^*.
\end{equation}
If we write $S'$ as $\lim S_i$ with strict \'etale affine transition morphisms,
then
\[
g_i^*\kappa_*
\simeq
\kappa_*g_i^*
\colon
\rC_\letale(S,\Lambda)\to \rC_\ket(S_i,\Lambda),
\]
where $g_i\colon S_i\to S$ is the induced strict \'etale morphism.
By taking limits on both sides,
we get
\[
g^*\kappa_*
\simeq
\kappa_*g^*
\colon
\rC_\letale(S,\Lambda)\to \rC_\ket(S',\Lambda).
\]
Combining this with \eqref{dim.5.2},
we get
\[
g^*\kappa_*
\simeq
\kappa_*g^*
\colon
\rD_\letale(S,\Lambda)\to \rD_\ket(S',\Lambda).
\]

Hence to show \eqref{dim.5.1},
replacing $S$ by $S'$,
we reduce to the case where $\ul{S}$ is strictly local.
Then $S$ has finite log \'etale and Kummer \'etale cohomological dimensions by Proposition \ref{dim.1}.
It suffices to show
\[
\hom_{\rD_\ket(S,\Lambda)}(\Lambda_U,\cF)
\xrightarrow{\simeq}
\hom_{\rD_\letale(S,\Lambda)}(\Lambda_U,\kappa^*\cF)
\]
for every $U\in S_\ket$.
By adjunction,
replacing $S$ by $U$,
it suffices to show
\[
R\Gamma_\ket(S,\cF)
\xrightarrow{\simeq}
R\Gamma_\letale(S,\kappa^*\cF).
\]
Both sides commute with colimits on $\cF$ by Proposition \ref{dim.6}.
Hence we reduce to the case where $\cF$ is concentrated in degree $0$.
This case is proven in \cite[Proposition 5.4(2)]{Nak2}.
\end{proof}

\begin{prop}
\label{dim.3}
Let $f\colon X\to S$ be a morphism of noetherian fs log schemes,
and let $\Lambda$ be a torsion commutative ring.
If $f$ is of finite type,
then $f_*\colon \rD_\letale(X)\to \rD_\letale(S,\Lambda)$ and $f_*\colon \rD_\ket(X)\to \rD_\ket(S,\Lambda)$ preserve colimits.
\end{prop}
\begin{proof}
Argue as in \cite[Lemme 4.2]{Ayo14} or \cite[Corollary 1.1.15]{CDetale},
but use Propositions \ref{dim.1} and \ref{dim.2}(2) also.
\end{proof}

For a log \'etale (resp.\ Kummer \'etale) morphism $f\colon X\to S$ of qcqs fs log schemes and a commutative ring $\Lambda$,
let
\[
f_\sharp \colon \rD_\letale(X,\Lambda)\to \rD_\letale(S,\Lambda) \text{ (resp.\ }f_\sharp \colon \rD_\ket(X,\Lambda)\to \rD_\ket(S,\Lambda))
\]
be the left adjoint to
\[
f^*\colon \rD_\letale(S,\Lambda)\to \rD_\letale(X,\Lambda) \text{ (resp.\ }f^*\colon \rD_\ket(S,\Lambda)\to \rD_\ket(X,\Lambda)).
\]

\begin{prop}
\label{dim.7}
Let $i\colon Z\to S$ be a strict closed immersion of noetherian fs log schemes,
let $j$ be its open complement,
and let $\Lambda$ be a torsion commutative ring.
Then we have a fiber sequence
\[
j_\sharp j^*\to \id \to i_*i^*\colon \rD_\letale(S,\Lambda)\to \rD_\letale(S,\Lambda).
\]
The analogous statement holds for the Kummer \'etale case too.
\end{prop}
\begin{proof}
We focus on the log \'etale case since the proofs are similar.

The diagram
\[
\begin{tikzcd}
&
j_\sharp j^* \ar[d]\ar[r]&
\id\ar[d]
\\
0\ar[r,"\simeq"]&
i_*i^*j_\sharp j^*\ar[r]&
i_*i^*
\end{tikzcd}
\]
yields a null sequence $j_\sharp j^*\to \id \to i_*i^*$.
By Proposition \ref{dim.3},
it suffices to show that $j_\sharp j^*\cF\to \cF \to i_*i^*\cF$ is a fiber sequence for every $\cF\in \rD_\letale(S,\Lambda)$ concentrated in degree $0$.
This is proven in \cite[5.2]{Nak2} (\cite[2.8.3]{Nak1} in the Kummer \'etale case).
\end{proof}

\begin{thm}[Kato--Nakayama]
\label{dim.4}
Let
\[
\begin{tikzcd}
X'\ar[d,"f'"']\ar[r,"g'"]&
X\ar[d,"f"]
\\
S'\ar[r,"g"]&
S
\end{tikzcd}
\]
be a cartesian square of noetherian fs log schemes.
Assume one of the following conditions:
\begin{enumerate}
\item[\textup{(1)}]
$f$ is proper, and $\Lambda$ is a torsion commutative ring.
\item[\textup{(2)}]
$f$ is of finite type,
$g$ is log smooth, and $\Lambda$ is a torsion commutative ring killed by an integer invertible on $S$.
\end{enumerate}
Then
\[
g^*f_*
\simeq
f_*'g'^*
\colon
\rD_\letale(X,\Lambda)
\to
\rD_\letale(S',\Lambda).
\]
If we further assume that either $f$ or $g$ is exact,
then
\[
g^*f_*
\simeq
f_*'g'^*
\colon
\rD_\ket(X,\Lambda)
\to
\rD_\ket(S',\Lambda).
\]
\end{thm}
\begin{proof}
For the log \'etale case,
Kato and Nakayama \cite[Theorems 6.1, 12.1]{Nak2} proved this when restricted to $\rD_\letale^+(X,\Lambda)$.
Since $\rD_\letale(X,\Lambda)$ is generated under colimits and shifts by $\rD_\letale^+(X,\Lambda)$,
Proposition \ref{dim.3} finishes the proof.

For the Kummer \'etale case,
use \cite[Theorem 5.1]{Nak1} and \cite[Theorem 12.7, Remark 12.7.1]{Nak2} instead.
\end{proof}

\section{Log \'etale motives}
\label{motive}

For a qcqs fs log scheme $S$ and a commutative ring $\Lambda$,
let $\lSm_S$ be the category of qcqs fs log schemes that are log smooth over $S$.
Let
\(
\DA_\letale^\eff(S,\Lambda)
\)
be the full subcategory of $\rD(\Sh_\letale(\lSm_S,\Lambda))$ spanned by the objects that are $\A^1$-invariant and invariant under verticalization.
Here,
an object $\cF$ of the derived $\infty$-category $\rD(\Sh_\letale(\lSm_S,\Lambda)$ is said to be \emph{invariant under verticalization} if $\cF(X)\xrightarrow{\simeq} \cF(X-\partial_S X)$ for every $X\in \lSm_S$,
where $X-\partial_S X$ is the largest open subscheme of $X$ that is vertical over $S$.
Let $\DA_\letale(S,\Lambda)$ be the $\P^1$-stabilization of $\DA_\letale^\eff(S,\Lambda)$.
For every $X\in \lSm_S$,
let $\Lambda_X$ denote the objects of $\DA_\letale^\eff(S,\Lambda)$ and $\DA_\letale(S,\Lambda)$ representable by $X$.
Let $\DA_\letale^\ex(S,\Lambda)$ be the full subcategory of $\DA_\letale(S,\Lambda)$ generated under colimits, shifts, and Tate twists by $\Sigma^\infty X_+$ for all exact $X\in \lSm_S$.

We have the localization functor
\[
L_\mot
\colon
\rD(\Sh_\letale(\lSm_S,\Lambda))
\to
\DA_\letale^\eff(S,\Lambda)
\]
with a fully faithful right adjoint $\iota_\mot$.
For $X\in \lSm_S$,
the canonical functor $\rho_X\colon X_\letale\to \lSm_S$ induces the restriction functor
\[
\rho_X^*\colon 
\rD(\Sh_\letale(\lSm_S,\Lambda))
\to
\rD_\letale(X,\Lambda)
\]
with a left adjoint $\rho_{X\sharp}$.

\begin{prop}
\label{DA.2}
Let $S$ be an fs log scheme,
and let $\Lambda$ be a commutative ring.
Then the functor
\[
\rho_{S\sharp}\colon \rD_\letale(S,\Lambda)\to \rD(\Sh_\letale(\lSm_S,\Lambda))
\]
is fully faithful.
\end{prop}
\begin{proof}
By \cite[\href{https://stacks.math.columbia.edu/tag/00XU}{Tag 00XU}]{stacks-project},
the functor $\rho_{S\sharp}\colon \Sh_\letale(S_\letale,\Lambda)\to \Sh_\letale(\lSm_S,\Lambda)$ is fully faithful and exact,
and its right adjoint $\rho_S^*\colon \Sh_\letale(\lSm_S,\Lambda)\to \Sh_\letale(S_\letale,\Lambda)$ is exact.
Hence the derived functor $\rho_{S\sharp}\colon \rD_\letale(S,\Lambda)\to \rD(\Sh_\letale(\lSm_S,\Lambda))$ is fully faithful.
\end{proof}

The composite functor
\begin{equation}
\label{DA.0.1}
\rho_{!}
\colon
\rD_\letale(S,\Lambda)
\xrightarrow{\rho_{S\sharp}}
\rD(\Sh_\letale(\lSm_S,\Lambda))
\xrightarrow{L_\mot}
\DA_\letale^\eff(S,\Lambda)
\xrightarrow{\Sigma^\infty}
\DA_\letale(S,\Lambda)
\end{equation}
will be used for our main result about log \'etale rigidity in \S \ref{rig}.

For the underlying scheme $\ul{S}$,
by \cite[Proposition 2.31]{logA1}\footnote{The reference considers only log \'etale sheaves, but the same proof works for log \'etale hypersheaves.},
we have equivalences of $\infty$-categories
\[
\DA_\letale^\eff(\ul{S},\Lambda)
\simeq
\DA_\et^\eff(\ul{S},\Lambda),
\;
\DA_\letale(\ul{S},\Lambda)
\simeq
\DA_\et(\ul{S},\Lambda).
\]

In the remaining part of this section,
we discuss the colimit preserving property for some functors without assuming finite log \'etale cohomological dimension.

\begin{prop}
\label{dim.8}
Let $S$ be an fs log scheme.
Then the restriction functor
\[
\rho_X^*\colon
\rD(\Sh_\letale(\lSm_S,\Lambda))
\to
\rD_\letale(X,\Lambda)
\]
preserves colimits for every $X\in \lSm_S$.
\end{prop}
\begin{proof}
Let $\cF:=\colim_{i\in I} \cF_i$ be a colimit in $\rD(\Sh_\letale(\lSm_S,\Lambda))$.
For every $U\in X_\letale$,
$\cF(U)$ is the evaluation at $U$ of the hypersheafification of
\[
U'\in U_{\letale}\mapsto \colim_{i\in I} (\cF_i(U')).
\]
Consider $\cG:=\colim_{i\in I} \cG_i$ in $\rD_\letale(X,\Lambda)$ with $\cG_i:=\rho_X^* \cF_i$.
As above,
$\cG(U)$ is the evaluation at $U$ of the hypersheafification of
\[
U'\in U_\letale \mapsto \colim_{i\in I} (\cG_i(U')).
\]
Since $\cF_i(U')\simeq \cG_i(U')$ for every $U'\in U_\letale$,
we conclude.
\end{proof}

\begin{prop}
\label{dim.9}
Let $S$ be a noetherian fs log scheme.
Then the inclusion
\[
\iota_{\mot} \colon \DA_\letale^\eff(S,\Lambda)
\to
\rD(\Sh_\letale(\lSm_S,\Lambda))
\]
preserves colimits.
\end{prop}
\begin{proof}
Let $\cF:=\colim_{i\in I} \cF_i$ be a colimit in $\DA_\letale^\eff(S,\Lambda)$.
Consider $\cG:=\colim_{i\in I} \cG_i$ in $\rD(\Sh_\letale(\lSm_S,\Lambda))$ with $\cG_i:=\iota_{\mot} \cF_i$.
We only need to show that $\cG$ is $\A^1$-invariant and invariant under verticalization since this implies $G\simeq \iota_\mot \cF$.

Let $U\in \lSm_S$,
and let $j\colon V:=U-\partial_S U\to U$ be the obvious open immersion.
Then we have
\begin{equation}
\label{dim.9.1}
\rho_U^*\cG_i
\simeq
j_*\rho_{V}^* \cG_i
\end{equation}
since
\[
(\rho_U^*\cG_i)(U')\simeq \cG_i(U')\simeq \cG_i(U'-\partial_S U')\simeq (j_*\rho_V^* \cG_i)(U')
\]
for every $U'\in U_\letale$,
where the second isomorphism holds since $\cG_i$ is invariant under verticalization,
and the third isomorphism is due to  \cite[Proposition 2.17]{logA1}          .
We also have
\begin{align*}
&\cG(U)\simeq (\rho_U^*\cG)(U)
\simeq (\colim_{i\in I} \rho_{U}^*\cG_i)(U)
\simeq (\colim_{i\in I} j_*\rho_{V}^*\cG_i)(U)
\\
\simeq &(j_*\colim_{i\in I}  \rho_{V}^* \cG_i)(U)
\simeq (\colim_{i\in I} \rho_{V}^*\cG_i)(V)
\simeq (\rho_{V}^*\colim_{i\in I} \cG_i)(V)\simeq \cG(V),
\end{align*}
where the second and sixth isomorphisms are due to Proposition \ref{dim.8},
the third isomorphism is due to \eqref{dim.9.1},
and the fourth isomorphism is due to Proposition \ref{dim.3}.
This shows that $\cG$ is invariant under verticalization.

The proof that $\cG$ is $\A^1$-invariant is similar.
\end{proof}

\begin{prop}
\label{dim.10}
Let $f\colon X\to S$ be a morphism of noetherian fs log schemes.
If $f$ is of finite type,
then $f_*\colon \DA_\letale^\eff(X,\Lambda)\to \DA_\letale^\eff(S,\Lambda)$ preserves colimits.
\end{prop}
\begin{proof}
The family of functors
\[
\{\rho_U^*:=\rho_U^*\iota_\mot\colon \DA_\letale^\eff(S,\Lambda)\to \rD_\letale(U,\Lambda)\}_{U\in \lSm_S}
\]
is conservative.
Hence by Propositions \ref{dim.8} and \ref{dim.9},
it suffices to show that $\rho_U^*f_*$ preserves colimits for every $U\in \lSm_S$.
Since $\rho_U^* f_*\simeq f_*\rho_{U\times_S X}^*$,
we conclude by Propositions \ref{dim.3} and \ref{dim.8}.
\end{proof}

\section{Poincar\'e duality for log \'etale morphisms}
\label{4}

In this section,
we prove Poincar\'e duality for log \'etale morphisms in Proposition \ref{logv.4}.
Nakayama \cite[Proposition in 7.1]{Nak1} proved this under the assumption that $\Lambda$ is noetherian.

We begin with a factorization result on separated log \'etale morphisms.
We refer to \cite[Definition II.2.3.1(2)]{Ogu} for the notion of neat charts.

\begin{prop}
\label{logv.10}
Let $f\colon X\to S$ be a separated log \'etale (resp.\ Kummer \'etale) morphism of qcqs fs log schemes.
Assume that $S$ has a chart $P$ neat at a point of $S$.
Then log \'etale locally on $X$,
there exists a factorization
\[
X\xrightarrow{j} Y\xrightarrow{p} S
\]
such that $j$ is an open immersion and $p$ is proper (resp.\ finite Kummer).
\end{prop}
\begin{proof}
We first work in the log \'etale case.

For a fan $\Sigma$,
let $\T_\Sigma$ be the associated fs log scheme in \cite[\S 2.2]{logSH}.
By \cite[Theorem III.2.6.7]{Ogu},
there exists a log blow-up $\T_\Sigma\to \A_P$ such that the pullback $X\times_{\A_P} \T_\Sigma\to S\times_{\A_P} \T_\Sigma$ is $\Q$-integral and and hence Kummer.

Consider the morphism of fans $\Sigma':=\Sigma \to \Sigma$ induced by multiplying an integer $n>0$ on the lattice of $\Sigma$.
By \cite[Theorem I.4.9.1]{Ogu},
there exists $n$ such that the pullback
\[
f'\colon X':=X\times_{\A_P} \T_{\Sigma'}\to S':=S\times_{\A_P} \T_{\Sigma'}
\]
is saturated and hence strict.
Zariski's main theorem implies that $f'$ admits a factorization $X'\to Y'\to S'$ such that the first morphism is an open immersion and the second morphism is strict finite.
To conclude,
observe that $X'\to X$ is a log \'etale covering and $S'\to S$ is proper.

For the Kummer \'etale case, skip the part with $\T_\Sigma$, and note that the multiplication by $n$ on $P$ induces a finite Kummer endomorphism on $\A_P$.
\end{proof}

\begin{const}
\label{logv.5}
Let $f\colon X\to S$ be a separated log \'etale morphism of qcqs fs log schemes,
and let $\Lambda$ be a torsion commutative ring.
Assume that $f$ admits a factorization
\[
X\xrightarrow{j} Y \xrightarrow{g} S
\]
such that $j$ is an open immersion and $g$ is proper.
Then we have an induced commutative diagram with cartesian squares
\[
\begin{tikzcd}
X\ar[r,"a"]\ar[rr,bend left=25,"b"]&
X\times_S X\ar[d,"f''"']\ar[r,"j'"]&
X\times_S Y\ar[d,"f'"]\ar[r,"g'"]&
X\ar[d,"f"]
\\
&
X\ar[r,"j"]&
Y\ar[r,"g"]&
S,
\end{tikzcd}
\]
where $a$ is the diagonal morphism.
By \cite[Lemma A.11.2]{logDM}\footnote{In this reference, the statement is written for fs log schemes admitting charts Zariski locally, but the same proof is applicable without this assumption},
$a$ is an open immersion.
Hence $b$ is an open immersion too.
Since $f$ and $g$ are separated,
we deduce that $a$ and $b$ are strict closed immersions.
In particular, we have
\[
b_\sharp\simeq b_*\colon \rD_\letale(X,\Lambda)\to \rD_\letale(X\times_S Y,\Lambda).
\]
Using this,
we have a composite natural transformation
\[
f_\sharp
\xrightarrow{\simeq}
f_\sharp g_*'b_*
\to
g_*f_\sharp' b_*
\xrightarrow{\simeq}
g_*f_\sharp' b_\sharp
\xrightarrow{\simeq}
g_*j_\sharp
\colon
\rD_\letale(X,\Lambda)
\to
\rD_\letale(S,\Lambda).
\]
\end{const}

We need the following log \'etale analogue of \cite[\href{https://stacks.math.columbia.edu/tag/04DQ}{Tag 04DQ}]{stacks-project}.

\begin{lem}
\label{logv.8}
Let $f\colon X\to S$ be a strict finite morphism of qcqs fs log schemes.
Then $f_*\colon \Sh(X_\letale,\Lambda)\to \Sh(S_\letale,\Lambda)$ is exact for every commutative ring $\Lambda$.
\end{lem}
\begin{proof}
Arguing as in
\cite[\href{https://stacks.math.columbia.edu/tag/04C2}{Tag 04C2}]{stacks-project},
it suffices to show the log \'etale analogue of the property (B) in \cite[\href{https://stacks.math.columbia.edu/tag/04C6}{Tag 04C6}]{stacks-project}.
Since the class of strict finite morphisms is closed under pullbacks,
it suffices to show the following:
Let $g\colon V\to X$ be a log \'etale covering of qcqs fs log schemes.
Then there exists a log \'etale covering $U\to S$ of qcqs fs log schemes such that we have $U\times_S X\simeq \amalg_{i\in I} U_i$ with the property that each $U_i\to X$ factors through $V$.

We can work log \'etale locally on $S$,
so we may assume that $S$ has a chart $P$.
Then $X$ has a chart $P$.
By \cite[Theorem III.2.6.7]{Ogu},
we may also assume that $g$ is $\Q$-integral and hence Kummer \'etale.
We can replace $V$ by its strict \'etale covering.
Hence by \cite[Theorem IV.3.3.1]{Ogu},
we may assume $V\simeq \amalg_{j\in J} V_j$ such that each $V_j\to X$ admits a chart $\theta_j\colon P\to Q_j$ such that $m:=\lvert \cok(\theta_j^\gp) \rvert$ is finite and invertible on $V_j$ and the induced morphism $V_j\to X\times_{\A_P} \A_{Q_j}$ is strict \'etale.

Let $g_j\colon V_j\to X$ be the induced morphism, which is an open morphism by \cite[Remark 5.7]{MR2604925}.
Apply \cite[\href{https://stacks.math.columbia.edu/tag/04DQ}{Tag 04DQ}]{stacks-project} to the Zariski covering $\{g_j(V_j)\to X\}_{j\in J}$ to obtain a strict \'etale covering $U\to S$ such that $U\times_S X\simeq \amalg_{i\in I} U_i$ with the property that $U_i\to X$ factors through $g_j(V_j)$ for some $j$.
Replacing $S$ by $U$ and $X$ by $U_i$,
we reduce to the case where $g_j$ is surjective for some $j\in J$.
Then it suffices to show the claim for the log \'etale covering $g_j$,
so we reduce to the case where $J=\{j\}$ and $V_j=V$.
We set $Q:=Q_j$ and $m:=m_j$.

Then $m$ is invertible on $X$.
Since $f(X)$ is closed,
we can replace $S$ by $S[1/m]$,
so we may assume that $m$ is invertible on $S$.
Then the projection $S\times_{\A_P} \A_Q\to S$ is a log \'etale covering by \cite[Corollary IV.3.1.10]{Ogu}.
Replacing $S$ by $S\times_{\A_P} \A_Q$,
we reduce to the case where $g$ is strict \'etale.
Then use the result for schemes \cite[\href{https://stacks.math.columbia.edu/tag/04DQ}{Tag 04DQ}]{stacks-project}.
\end{proof}

We first treat a special case of Proposition \ref{logv.4} as follows.

\begin{lem}
\label{logv.3}
Let $f\colon X\to S$ be a separated strict \'etale morphism of qcqs fs log schemes,
and let $\Lambda$ be a torsion commutative ring.
Then the morphism
\[
f_\sharp \cF\to f_! \cF
\]
in $\rD_\letale(S,\Lambda)$ 
is an isomorphism for every sheaf $\cF$ of $\Lambda$-modules on $X_\letale$.
The analogous statement holds in the Kummer \'etale case too.
\end{lem}
\begin{proof}
We focus on the log \'etale case since the proofs are similar.
We can work strict \'etale locally on $S$,
so we may assume that $S$ has a chart $P$ neat at a point of $S$.

See \cite[Lemma in Chapter 6, p.\ 294]{Ayoetale} for a similar proof in the non-log case.
Let us work at the level of sheaves of $\Lambda$-modules and add the notation $R$ for right derived functors.
By Zariski's main theorem,
we have a factorization $X\xrightarrow{j} Y \xrightarrow{g} S$ such that $j$ is an open immersion and $g$ is strict finite.
The statement is $f_\sharp \cF\simeq Rf_! \cF$.
Since $g_*$ is exact by Lemma \ref{logv.8},
we have $f_!:=g_* j_\sharp \simeq Rg_* j_\sharp \simeq Rf_!$.

The \emph{additive topology} on the category of qcqs fs log schemes is defined to be the coarsest topology such that for every finite family of qcqs fs log schemes $\{X_i\}_{i\in I}$, $\{X_i\to \amalg_{i\in I}X_i\}_{i\in I}$ is a covering family.
A standard argument shows that $f_\sharp \cF$ is the log \'etale sheafification of the additive sheaf $f_\sharp^\add \cF$ such that
\begin{equation}
\label{logv.3.1}
f_\sharp^\add \cF(U)
\simeq
\bigoplus_{
  \begin{tikzpicture}[
    baseline=-0.5ex,
    >=stealth,
    scale=0.55,
    every node/.style={font=\scriptsize, inner sep=1pt}
  ]
    \node (X) at (0,0.8) {$U$};
    \node (Y) at (1.0,0.8) {$X$};
    \node (Z) at (0.5,0) {$S$};

    \draw[->] (X) -- (Y);
    \draw[->] (X) -- (Z);
    \draw[->] (Y) -- (Z);
  \end{tikzpicture}
}
\cF(U\to X)
\end{equation}
for every connected $U\in S_\letale$.
On the other hand,
we have
\begin{equation}
\label{logv.3.2}
g_* f_\sharp'^\add b_* \cF(U)
\simeq
\bigoplus_{
  \begin{tikzpicture}[
    baseline=-0.5ex,
    >=stealth,
    scale=0.55,
    every node/.style={font=\scriptsize, inner sep=1pt}
  ]
    \node (V) at (-2.5,0.8) {$V$};
    \node (X) at (-0.75,0.8) {$U\times_S Y$};
    \node (Y) at (1.75,0.8) {$X\times_S Y$};
    \node (Z) at (0.5,0) {$Y$};

    \draw[->] (V) -- (X);
    \draw[->] (X) -- (Y);
    \draw[->] (X) -- (Z);
    \draw[->] (Y) -- (Z);
  \end{tikzpicture}
}
\cF(V\times_{X\times_S Y}X \to X),
\end{equation}
where $V$ is a connected components of $U\times_S Y$.
We also have
\begin{equation}
\label{logv.3.3}
g_*f_\sharp'^\add b_\sharp^\add \cF(U)
\simeq
g_*j_\sharp^\add\cF(U)
\simeq
\bigoplus_{
  \begin{tikzpicture}[
    baseline=-0.5ex,
    >=stealth,
    scale=0.55,
    every node/.style={font=\scriptsize, inner sep=1pt}
  ]
    \node (V) at (-2.25,0.8) {$V$};
    \node (X) at (-0.5,0.8) {$U\times_S Y$};
    \node (Y) at (1.5,0.8) {$X$};
    \node (Z) at (0.5,0) {$Y$};

    \draw[->] (V) -- (X);
    \draw[->] (X) -- (Y);
    \draw[->] (X) -- (Z);
    \draw[->] (Y) -- (Z);
  \end{tikzpicture}
}
\cF(V\to X),
\end{equation}
where $V$ is a connected components of $U\times_S Y$.

The morphism \eqref{logv.3.1}$\to$\eqref{logv.3.2} is obtained by sending $U$ to the graph morphism $U\to U\times_S Y$,
which is possible since $U$ is a connected component of $U\times_S Y$ and we have $U\times_{X\times_S Y} X \simeq U$.
The isomorphism \eqref{logv.3.3}$\to$\eqref{logv.3.2} is obtained by sending $V$ to $V$,
which is possible since $V\times_{X\times_S Y} X\simeq V$.
Hence the morphism \eqref{logv.3.1}$\to$\eqref{logv.3.3} is obtained by sending $U$ to the graph morphism $U\to U\times_S Y$.

It suffices to show that the morphism of the stalks of the strict \'etale sheafifications
\begin{equation}
\label{logv.3.4}
(a_\setale f_\sharp^\add \cF)_{\ol{s}}
\to
(a_\setale g_*j_\sharp^\add\cF)_{\ol{s}}
\end{equation}
is an isomorphism at any geometric point $\ol{s}$ of $S$.
We have
\[
(a_\setale f_\sharp^\add \cF)_{\ol{s}}
\simeq
\bigoplus_{
  \begin{tikzpicture}[
    baseline=-0.5ex,
    >=stealth,
    scale=0.55,
    every node/.style={font=\scriptsize, inner sep=1pt}
  ]
    \node (X) at (0,0.8) {$\ol{s}$};
    \node (Y) at (1.0,0.8) {$X$};
    \node (Z) at (0.5,0) {$S$};

    \draw[->] (X) -- (Y);
    \draw[->] (X) -- (Z);
    \draw[->] (Y) -- (Z);
  \end{tikzpicture}
}
\cF_{\ol{s}\to X}.
\]
On the other hand,
since the underlying scheme of the strict henselization $S_{\ol{s}}$ is strictly local,
the strict finite morphism $S_{\ol{s}}\times_S Y\to S_{\ol{s}}$ is a finite disjoint union of the spectra of local rings with the induced log structures.
Moreover, for any connected component $V$ of $S_{\ol{s}}\times_S Y$ mapping to $X$,
the induced morphism $V\to S_{\ol{s}}$ is \'etale and hence admits a section $i\colon S_{\ol{s}}\to V$.
Then the composite $c\colon S_{\ol{s}}\xrightarrow{i} V\to S_{\ol{s}} \times_S X$ is the graph morphism of the composite $d\colon S_{\ol{s}}\to V\to X$.
Hence $c$ exhibits $S_{\ol{s}}$ as a connected component of $S_{\ol{s}} \times_S X$, so $i$ is an isomorphism.
It follows that the assignment
\[
(V\to S_{\ol{s}}\times_S Y)
\mapsto
(d\colon S_{\ol{s}}\to X)
\]
yields an inverse of \eqref{logv.3.4}.
\end{proof}

Now, we prove the main result in this section.

\begin{prop}
\label{logv.4}
Let $f\colon X\to S$ be a separated log \'etale (resp.\ Kummer \'etale) morphism of qcqs fs log schemes,
and let $\Lambda$ be a torsion commutative ring.
Then the morphism
\[
f_\sharp \cF\to f_! \cF
\]
in $\rD_\letale(S,\Lambda)$ (resp.\ $\rD_\ket(S,\Lambda)$)
is an isomorphism for every $\cF\in \rD_\letale(S,\Lambda)$ (resp.\ $\rD_\ket(S,\Lambda)$) concentrated in degree $0$.
\end{prop}
\begin{proof}
We focus on the log \'etale case since the proofs are similar.

The question is log \'etale local on $S$.
Hence by \cite[Theorem III.2.6.7]{Ogu},
we may assume that $f$ is Kummer \'etale.
We may also assume that $S$ has a chart $P$ neat at a point of $S$.
By \cite[Remark 5.7]{MR2604925},
the image $U$ of $f$ is open.
Let $j\colon U\to S$ be the obvious open immersion,
and let $i\colon Z\to S$ be a closed complement of $j$.
By \cite[5.2]{Nak2} (\cite[2.8.3]{Nak1} in the Kummer \'etale case),
it suffices to show
\[
i^* f_\sharp \cF\simeq i^* f_!\cF,
\;
j^*f_\sharp \cF \simeq j^* f_!\cF.
\]
Since $Z\times_S U\simeq 0$,
we have $i^*f_!\cF\simeq 0$ by \cite[Theorem 5.1]{Nak1}.
We also have $i^*f_\sharp \cF\simeq 0$.
Hence replacing $S$ by $U$,
we reduce to the case where $f$ is surjective.

By \cite[Theorem I.4.9.1]{Ogu},
there exists an integer $n>0$ such that $n$ is invertible on $X$ and the pullback
\[
X\times_{\A_P} \A_{P'}\to S\times_{\A_P} \A_{P'}
\]
is saturated and hence strict, where $P\to P=:P'$ is the multiplication by $n$.
Since $f$ is surjective,
$n$ is invertible on $S$.
Hence the projection $S\times_{\A_P} \A_{P'}\to S$ is Kummer \'etale.
Since the question is log \'etale local on $S$,
we reduce to the case where $f$ is strict \'etale.

Lemma \ref{logv.3} finishes the proof.
\end{proof}

Here is an application of Proposition \ref{logv.4}.

\begin{prop}
\label{logv.1}
Let $S$ be a qcqs fs log scheme,
and let $\Lambda$ be a torsion commutative ring.
Assume that $S$ has a chart neat at a point of $S$.
Then $\rD_\letale(S,\Lambda)$ (resp.\ $\rD_\ket(S,\Lambda)$) is generated under colimits and shifts by the objects of the form $v_* \Lambda_V$ for all proper (resp.\ finite Kummer) morphisms $v\colon V\to S$.
\end{prop}
\begin{proof}
We focus on the log \'etale case since the proofs are similar.

Let $\cC$ be the full subcategory generated under colimits and shifts by such objects,
and let $f\colon X\to S$ be a log \'etale morphism.
It suffices to show $f_\sharp \Lambda_X\in \cC$ log \'etale locally on $X$.
Hence by Proposition \ref{logv.10},
we may assume that $f$ admits a factorization $X\xrightarrow{j} Y\xrightarrow{g} S$ such that $j$ is an open immersion and $g$ is proper.

By \cite[5.2]{Nak2} (\cite[2.8.3]{Nak1} in the Kummer \'etale case),
we have a fiber sequence
\[
f_! \Lambda_X \to g_* \Lambda_Y \to g_* i_* \Lambda_{Y-X},
\]
where $i\colon Y-X\to Y$ is the induced strict closed immersion.
To conclude,
observe that we have $f_\sharp\Lambda_X\simeq f_!\Lambda_X$ by Proposition \ref{logv.4}.
\end{proof}

\section{\texorpdfstring{$\A^1$}{A1}-invariance and \texorpdfstring{$\square$}{box}-invariance}
\label{5}

Recall from \cite[Theorem 1.3.2]{CDetale} that for a noetherian scheme $S$ and a torsion commutative ring $\Lambda$ killed by an integer invertible on $S$,
the functor
\[
p^* \colon \rD_\et(S,\Lambda)\to \rD_\et(S\times \A^1,\Lambda)
\]
is fully faithful.
We extend this $\A^1$-invariance property to the logarithmic case as follows.

\begin{thm}
\label{A1.1}
Let $S$ be a noetherian fs log scheme,
and let $\Lambda$ be a torsion commutative ring killed by an integer invertible on $S$.
Then the functor
\[
p^*\colon \rD_\letale(S,\Lambda)
\to
\rD_\letale(S\times \A^1,\Lambda)
\]
is fully faithful,
where $p\colon S\times \A^1\to S$ is the projection.
\end{thm}
\begin{proof}
It suffices to show
\[
\hom_{\rD_\letale(S,\Lambda)}(\Lambda_U,\cF)
\simeq
\hom_{\rD_\letale(S\times \A^1,\Lambda)}(p^*\Lambda_U,p^*\cF)
\]
for all $U\in S_{\letale}$ and $\cF\in \rD_{\letale}(S)$.
Replacing $S$ by $U$,
it suffices to show
\[
R\Gamma_{\letale}(S,\cF)
\simeq
R\Gamma_{\letale}(S\times \A^1,p^*\cF).
\]
This is equivalent to
\[
R\Gamma_{\letale}(\ul{S},r_* \cF)
\simeq
R\Gamma_{\letale}(\ul{S}\times \A^1,r_*'p^*\cF),
\]
where $r\colon S\to \ul{S}$ and $r'\colon S\times \A^1\to \ul{S}\times \A^1$ are the morphisms removing the log structures.
Since $r$ is proper,
Theorem \ref{dim.4}(1) yields
\[
R\Gamma_{\letale}(\ul{S}\times \A^1,r_*'p^*\cF)
\simeq
R\Gamma_{\letale}(\ul{S}\times \A^1,\ul{p}^* r_* \cF),
\]
so it suffices to show
\[
R\Gamma_{\letale}(\ul{S},\cG)
\simeq
R\Gamma_{\letale}(\ul{S}\times \A^1,\cG\vert_{\ul{S}\times \A^1})
\]
with $\cG:=r_*\cF$.
The log \'etale cohomology agrees with the usual \'etale cohomology for schemes with trivial log structure.
Hence we conclude by \cite[Theorem 1.3.2]{CDetale}.
\end{proof}

We will use the notation $\square:=(\P^1,\infty)$,
whose underlying scheme is $\P^1$ with the compactifying log structure associated with $\A^1\to \P^1$.
The following result generalizes \cite[Theorem 9.1.5]{logSH}.

\begin{thm}
\label{A1.2}
Let $S$ be a noetherian fs log scheme,
let $\Lambda$ be a torsion commutative ring,
and let $\cF\in \rD_{\letale}(S,\Lambda)$.
Then the functor
\[
p^*\colon \rD_\letale(S,\Lambda)
\to
\rD_\letale(S\times \square,\Lambda)
\]
is fully faithful,
where $p\colon S\times \square\to S$ is the projection.
\end{thm}
\begin{proof}
As in the proof of Theorem \ref{A1.1},
we reduce to showing
\[
R\Gamma_\letale(S,\cF)\simeq R\Gamma_\letale(S\times \square,\cF\vert_{S\times \square})
\]
under the assumption that $S$ has the trivial log structure.
For this,
it suffices to show
\[
\cF\simeq p_*p^*\cF.
\]
By Proposition \ref{dim.2}(2), we only need to consider the case where $S$ is strictly local.
Use \cite[Proposition 3.24]{Ayo14} and Theorem \ref{dim.4}(1) to reduce to the case where $S$ is the spectrum of a field.
Use Proposition \ref{dim.2}(2) again to further reduce to the case where $S$ is the spectrum of a separably closed field.
Proposition \ref{dim.3} allows us to reduce to the case where $\cF$ is concentrated in degree $0$.
Then $\cF$ is a constant sheaf.
We finish the proof by Proposition \ref{dim.5} and \cite[Theorem 9.1.5]{logSH}.
\end{proof}

\section{Invariance under virtual isomorphisms}
\label{6}

Recall from \cite[Definition 3.2.1]{logsix} that a morphism of fs log schemes $f\colon X\to S$ is called a \emph{virtual isomorphism} if $\ul{f}$ is an isomorphism and $\ol{\cM}_{S}^\gp \simeq \ol{\cM}_{X}^\gp$.

\begin{thm}
\label{virtual.1}
Let $f\colon X\to S$ be a virtual isomorphism of noetherian fs log schemes,
and let $\cF\in \rD_{\ket}(S,\Lambda)$.
Then
\[
R\Gamma_\ket(S,\cF)
\simeq
R\Gamma_\ket(X,\cF\vert_X).
\]
\end{thm}
\begin{proof}
Let $g\colon S\to \ul{S}$ be the morphism removing the log structure.
It suffices to show
\[
g_* \simeq g_*f_*f^*.
\]
Note that $g$ and $gf$ are proper.
As in the proof of Theorem \ref{A1.2},
we reduce to the case where $\ul{S}$ is the spectrum of a separably closed field of characteristic $p$.

Then all Kummer \'etale coverings of $S$ and $X$ are finite.
Hence by \cite[4.6]{Ill},
it suffices to show
\[
\pi_1^{\log}(S,\widetilde{s})
\simeq
\pi_1^{\log}(X,\widetilde{x}),
\]
where $\widetilde{s}$ and $\widetilde{x}$ are log geometric points of $S$ and $X$.
This is equivalent to
\[
\Hom(P^\gp,\prod_{\ell\neq p} \Z_\ell(1))
\simeq
\Hom(Q^\gp,\prod_{\ell\neq p} \Z_\ell(1))
\]
due to the computation in \cite[Example 4.7(a)]{Ill},
where $P:=\ol{\cM}_S(S)$ and $Q:=\ol{\cM}_X(X)$.
This is clear since $P^\gp\to Q^\gp$ is an isomorphism by the assumption that $f$ is a virtual isomorphism.
\end{proof}

\section{Log v-descent}
\label{7}

Let us review the log v-topology \cite{logv} introduced by Opdan, Park, and {\O}stv{\ae}r.
A \emph{log valuation ring} $(V,P)$ is a valuation ring $V$ with a valuative monoid $P$ and a logarithmic map $P\to V$.
A morphism of qcqs fs log schemes $f\colon X\to S$ is called a \emph{log v-cover} if for every log valuation ring $(V,P)$ with a morphism $\Spec(V,P)\to S$,
there exists a map of log valuation rings $(V,P)\to (W,Q)$ with $V\to W$ injective and local and a commutative square
\[
\begin{tikzcd}
\Spec(W,Q)\ar[d]\ar[r]&
X\ar[d,"f"]&
\\
\Spec(V,P)\ar[r]&
S.
\end{tikzcd}
\]
The \emph{log v-topology} is the topology on the category of qcqs fs log schemes generated by the finite covering families $\{U_i\to X\}_{i\in I}$ such that $\amalg_{i\in I} U_i\to X$ is a log v-cover.

Note that the presheaf
\[
X\mapsto R\Gamma_\letale(X,\Lambda_X)
\]
on the category of qcqs fs log schemes satisfies log v-descent by \cite[Theorem 4.10]{logv} for every torsion commutative ring $\Lambda$.

By \cite[Proposition 3.24]{logv},
the log v-topology is finer than the log fppf-topology,
which is the topology on the category of fs log schemes generated by finite covering families  $\{U_i\to X\}_{i\in I}$ such that $\amalg_{i\in I} U_i\to X$ is a log fppf-cover in the sense that it is universally surjective, finitely presented, and log flat.
Working log fppf-locally is helpful due to the following result:

\begin{prop}
\label{logv.11}
Let $f\colon X\to S$ be a morphism of qcqs fs log schemes.
Then there exists a log fppf-cover $S'\to S$ such that the pullback $X\times_S S'\to S'$ is saturated.
\end{prop}
\begin{proof}
By \cite[Theorem III.2.6.7]{Ogu},
we reduce to the case where $f$ is $\Q$-integral since any log blow-up is a log \'etale cover.
We may also assume that $S$ admits a chart $P$.
Use \cite[Theorem I.4.9.1]{Ogu} to find an integer $n>0$ such that the pullback $X\times_{\A_P}\A_{P'}\to S\times_{\A_P} \A_{P'}$ is saturated, where $P\to P=:P'$ is the multiplication by $n$.
To conclude,
observe that the induced morphism $\A_{P'}\to \A_P$ is a log fppf-cover.
\end{proof}

Recall from \cite[Definition 2.1]{regGysin} that the \emph{log cdh-topology} on the category of qcqs fs log schemes is the finest topology such that it is finer than the strict Nisnevich topology and $Z\amalg X\to S$ is a covering for every cartesian square
\begin{equation}
\label{logv.7.1}
\begin{tikzcd}
W\ar[d]\ar[r]&
X\ar[d,"f"]
\\
Z\ar[r,"i"]&
S
\end{tikzcd}
\end{equation}
such that $i$ is a strict closed immersion, $f$ is proper, and $f^{-1}(S-Z)\xrightarrow{\simeq} S-Z$.

By \cite[Theorem 3.23, Proposition 3.25]{logv},
the log v-topology is finer than the log cdh-topology.

\begin{thm}
\label{logv.7}
Let $S$ be a noetherian fs log scheme,
and let $\Lambda$ be a torsion commutative ring killed by an integer invertible on $S$.
Consider a square of the form \eqref{logv.7.1}.
Then the induced square
\[
\begin{tikzcd}
\id\ar[d]\ar[r]&
i_*i^*\ar[d]
\\
f_*f^*\ar[r]&
g_*g^*
\end{tikzcd}
\colon
\rD_\letale(S,\Lambda)
\to
\rD_\letale(S,\Lambda)
\]
is cartesian,
where $g\colon W\to S$ is the composite morphism.
The analogous statement holds in the Kummer \'etale case too.
\end{thm}
\begin{proof}
We focus on the log \'etale case since the proofs are similar.
We can work strict \'etale locally on $S$,
so we may assume that $S$ has a chart neat at a point of $S$.

By Propositions \ref{dim.3} and \ref{logv.1},
it suffices to show that
\[
\begin{tikzcd}
v_*\Lambda_V\ar[d]\ar[r]&
i_*i^*v_*\Lambda_V\ar[d]
\\
f_*f^*v_*\Lambda_V\ar[r]&
g_*g^*v_*\Lambda_V
\end{tikzcd}
\]
is cartesian for every proper morphism $V\to S$.
Replace $S$ by $V$ and use Theorem \ref{dim.4}(1) to reduce to the case where $V=S$.
Then it suffices to show that 
\[
\begin{tikzcd}
\hom_{\rD_\letale(S,\Lambda)}(\Lambda_U,\Lambda_S)\ar[d]\ar[r]&
\hom_{\rD_\letale(S,\Lambda)}(\Lambda_U,i_*i^*\Lambda_S)\ar[d]
\\
\hom_{\rD_\letale(S,\Lambda)}(\Lambda_U,f_*f^*\Lambda_S)\ar[r]&
\hom_{\rD_\letale(S,\Lambda)}(\Lambda_U,g_*g^*\Lambda_S)
\end{tikzcd}
\]
is cartesian for every $U\in S_\letale$.
Replace $S$ by $U$ and use Theorem \ref{dim.4} to reduce to the case where $U=S$.
As noted above,
the log v-topology is finer than the log cdh-topology.
Hence \cite[Remark 2.2]{regGysin} and \cite[Theorem 4.10]{logv} finish the proof.
\end{proof}

\section{Invariance under verticalization}
\label{8}

For any cartesian square of fs log schemes
\[
\begin{tikzcd}
X'\ar[d]\ar[r]&
X\ar[d]
\\
S'\ar[r]&
S,
\end{tikzcd}
\]
by \cite[Proposition 2.17]{logA1}, we have
\begin{equation}
\label{ver.0.1}
(X-\partial_S X)\times_S S'
\simeq
X'-\partial_{S'} X'.
\end{equation}

For a monoid $P$ and an ideal $I$ of $P$,
we will use the notation
\[
\A_{P,I}
:=
\Spec(P\mapsto \Z[\{x^p\}_{p\in P}]/(x^i)_{i\in I}),
\]
where the log structure homomorphism sends $p\in P$ to $x^p$.

Now, we prove the main technical result in this paper,
which uses many results from earlier sections.

\begin{thm}
\label{ver.1}
Let $f\colon X\to S$ be a log smooth morphism of noetherian fs log schemes,
and let $j\colon X-\partial_S X \to X$ be the obvious open immersion,
and let $S$ be a torsion commutative ring killed by an integer invertible on $S$.
Then the natural transformation
\[
f_*f^*
\xrightarrow{\ad}
f_*j_*j^*f^*
\colon
\rD_{\letale}(S,\Lambda)\to \rD_{\letale}(S,\Lambda)
\]
is an isomorphism.
\end{thm}
\begin{proof}
Step 1. \emph{Usual reduction.}
The claim is trivial if $j$ is an isomorphism, so assume that $j$ is not an isomorphism.
The question is strict \'etale local on $S$,
so we may assume that $S$ has a chart neat at a point of $S$.
Proposition \ref{dim.3} implies that $f_*$ and $f_*j_*$ preserve colimits.
Since $\rD_\letale(S,\Lambda)$ is generated under colimits and shifts by $v_*\Lambda_V$ for all proper $v\colon V\to S$, by Proposition \ref{logv.1},
it suffices to show
\[
f_*f^*v_*\Lambda_V\simeq f_*j_*j^*f^*v_*\Lambda_V.
\]
Replace $S$ by $V$ and use Theorem \ref{dim.4} and \eqref{ver.0.1} to reduce to the case where $V=S$.
Then it suffices to show
\[
\hom_{\rD_{\letale}(S,\Lambda)}(\Lambda_U,f_*f^*\Lambda_S)
\simeq
\hom_{\rD_{\letale}(S,\Lambda)}(\Lambda_U,f_*j_*j^*f^*\Lambda_S)
\]
for all $U\in S_\letale$.
This can be written as
\[
R\Gamma_\letale(U\times_S X,\Lambda_{U\times_S X})
\simeq
R\Gamma_\letale(U\times_S (X-\partial_S X),\Lambda_{U\times_S (X-\partial_S X)}).
\]
Replacing $S$ by $U$ and using \eqref{ver.0.1},
we reduce to showing
\[
R\Gamma_\letale(X,\Lambda_X)
\simeq
R\Gamma_\letale(X-\partial_S X,\Lambda_{X-\partial_S X}).
\]

Step 2. \emph{Log v-locality on $S$}.
We can work log v-locally on $S$ by \cite[Theorem 4.10]{logv} and \eqref{ver.0.1}.
By Proposition \ref{logv.11},
we may assume that $f$ is saturated.
We may also assume that $S$ has a chart $P$ neat at a point of $S$.
We can also work strict \'etale locally on $X$.
Together with \cite[Theorem IV.3.3.1]{Ogu},
we may assume that $f$ has a chart $\theta\colon P\to Q$ such that $Q$ is sharp, $\theta$ is saturated, and the induced morphism $X\to S\times_{\A_P} \A_Q$ is strict \'etale.

Let $\cF:=\Lambda_S\in \rD_\ket(S,\Lambda)$.
By Proposition \ref{dim.5},
it suffices to show
\[
R\Gamma_\ket(X,\cF\vert_X)
\simeq
R\Gamma_\ket(X-\partial_S X,\cF\vert_{X-\partial_S X}).
\]

Step 3. \emph{Stratification.}
If $\{g_i\colon S_i\to S\}_{i\in I}$ is a stratification, then Proposition \ref{dim.7} implies that $\Lambda_S$ is contained in the full subcategory of $\rD_\ket(S,\Lambda)$ generated under finite colimits by the essential images of $g_{i*}\colon \rD_\ket(S_i,\Lambda)\to \rD_\ket(S,\Lambda)$.
Applying Theorem \ref{dim.4}(2) to $f^*$ and $j^*f^*$,
it suffices to show the claim for each $S_i$ with arbitrary $\cF$ instead of $\Lambda_S$.
Hence
by \cite[Propositions 3.1.1, 3.1.4]{logsix},
it suffices to show
\[
R\Gamma_\ket(X,\cF\vert_X)
\simeq
R\Gamma_\ket(X-\partial_S X,\cF\vert_{X-\partial_S X})
\]
for every $\cF\in \rD_\ket(S,\Lambda)$
under the assumption $S\simeq \ul{S} \times \pt_P$.

Consider $Z:=X\times_{\A_Q} \pt_Q$,
which is contained in $\partial_S X$.
By induction on the rank of $Q^\gp$,
it suffices to show
\begin{equation}
\label{ver.1.1}
R\Gamma_\ket(X,\cF\vert_X)
\simeq
R\Gamma_\ket(X-Z,\cF\vert_{X-Z}).
\end{equation}

Step 4. \emph{A strict closed cover.}
Let $G_1,\ldots,G_r$ be the maximal $\theta$-critical faces of $Q$.
For any nonempty subset $I:=\{i_1,\ldots,i_s\}\subset \{1,\ldots,r\}$,
we set $G_I:=G_{i_1}\cap \cdots \cap G_{i_s}$.
By \cite[Theorem I.4.8.14(7)]{Ogu},
\[
\{ \A_{Q,Q-G_1},\ldots, \A_{Q,Q-G_r}\}
\]
is a strict closed cover of $\pt_P\times_{\A_P} \A_Q$.
We set $X_I:=X\times_{\A_Q} \A_{Q,Q-G_I}$.
Then by Theorem \ref{logv.7},
we have
\begin{gather}
\label{ver.1.2}
\lim_I R\Gamma_\ket(X_I,\cF\vert_{X_I})
\simeq
R\Gamma_\ket(X,\cF\vert_X),
\\
\label{ver.1.3}
\lim_I R\Gamma_\ket(X_I-Z,\cF\vert_{X_I-Z})
\simeq
R\Gamma_\ket(X-Z,\cF\vert_{X-Z}),
\end{gather}
where the limits run over the nonempty subsets $I$ of $\{1,\ldots,r\}$.
We will show
\begin{equation}
\label{ver.1.6}
R\Gamma_\ket(X_I,\cF\vert_{X_I})
\simeq
R\Gamma_\ket(X_I-Z,\cF\vert_{X_I-Z})
\end{equation}
below if $X_I\neq Z$.
Assume this for the moment.
Then
\begin{equation}
\label{ver.1.4}
\begin{split}
&\lim_I \fib(
R\Gamma_\ket(X_I,\cF\vert_{X_I})
\to
R\Gamma_\ket(X_I-Z,\cF\vert_{X_I-Z})
)
\\
\simeq &
\lim_I \fib(R\Gamma_\ket(Z,\cF\vert_Z)  \to R\Gamma_\ket(Z,\cF\vert_Z) \otimes_{\Z} \delta_I),
\end{split}
\end{equation}
where
\[
\delta_I=\left\{
\begin{array}{ll}
\Z & \text{if $X_I\neq Z$},
\\
0 & \text{if $X_I=Z$}.
\end{array}
\right.
\]
The $\R$-cone $G_I\otimes \R$ is contractible,
and $G_I\otimes \R-0$ is contractible if and only if $G_I\neq 0$.
Hence we have the following computations of singular cohomology complexes
\[
\Z\simeq R\Gamma_\sing(G_I\otimes \R),
\;
\delta_I \simeq R\Gamma_\sing(G_I\otimes \R-0).
\]
Consider $K:=G_1\cup \cdots \cup G_r$.
Since $\{G_1\otimes \R,\ldots,G_r\otimes \R\}$ is a closed cover of $K\otimes \R$,
we have
\[
\lim_I R\Gamma_\sing(G_I \otimes \R)
\simeq
R\Gamma(K \otimes \R),
\;
\lim_I R\Gamma_\sing(G_I \otimes \R-0)
\simeq
R\Gamma_\sing(K\otimes \R-0).
\]
Hence
\begin{equation}
\label{ver.1.5}
\lim_I \fib(\Z \to \delta_I)
\simeq
\fib(R\Gamma_\sing(K\otimes \R)\to R\Gamma_\sing(K\otimes \R-0)).
\end{equation}
By \cite[Theorem I.4.8.14(7)]{Ogu}, we have a bijection $P\times K\xrightarrow{\simeq} Q$.
The quotient $Q/P$ in the category of integral monoids is the image of $Q$ in $Q^\gp/P^\gp$.
Hence we have a bijection $K\simeq Q/P$.
Since $\theta$ is not vertical,
$Q/P$ is not a group by definition.
Hence $(Q/P)\otimes \R$ is not a vector space,
so $0$ is in the boundary of $(Q/P)\otimes \R$.
It follows that $0$ is in the boundary of $K\otimes \R$,
so the inclusion $K\otimes \R-0\to K\otimes R$ is a homotopy equivalence.
Together with \eqref{ver.1.2}, \eqref{ver.1.3}, \eqref{ver.1.4}, and \eqref{ver.1.5},
we deduce \eqref{ver.1.1}.

Step 5.
\emph{Log structure change.}
It remains to show \eqref{ver.1.6} if $X_I\neq Z$.
In this case,
$G_I\neq 0$.
Choose an element $i\in I$.
Then the induced map $P^\gp\oplus G_i^\gp\to Q^\gp$ is an isomorphism by \cite[Theorem I.4.8.14(7)]{Ogu}.
Toric resolution of singularities \cite[Theorem 11.1.9]{CLStoric} yields a map $\N^s\to G_i/G_I$ such that $\Z^s\xrightarrow{\simeq} (G_i/G_I)^\gp$.
Choose elements $x_1,\ldots,x_s\in G_i$ whose images in $G_i/G_I$ are the generators of $\N^s$.
Let $P'$ be the submonoid of $Q$ generated by $P$ and $x_1,\ldots,x_s$.
Then the induced map $P'^\gp \oplus G_I^\gp \to Q^\gp$ is an isomorphism.
Hence we have a virtual isomorphism
\[
\A_{Q,Q-G_I}
\to
\pt_{P'} \times \A_{G_I}
\]
since the underlying schemes of the domain and codomain are identified with $\ul{\A_{G_I}}$.
There exists a unique strict \'etale morphism $Y\to \ul{S}\times \pt_{P'}\times \A_{G_I}$ that pulled back to the strict \'etale morphism $X_I\to S\times_{\A_P} \A_{Q,Q-G_I}$.
Then the induced morphism $X_I\to Y$ is again a virtual isomorphism.
By Theorem \ref{virtual.1},
it suffices to show
\[
R\Gamma_\ket(Y,\cF\vert_Y)
\simeq
R\Gamma_\ket(Y-W,\cF\vert_{Y-W}),
\]
where $W:=Y\times_{\A_{G_I}} \pt_{G_I}$.

There exist unique log smooth morphisms $Y_0,Y_0-W_0\to \ul{S}$ that pulled back to the log smooth morphisms $Y,Y-W\to \ul{S}\times \pt_{P'}$.
Apply Lemma \ref{ver.2} below to these morphisms and use Proposition \ref{dim.5} to finish the proof.
\end{proof}

\begin{lem}
\label{ver.2}
Let $S$ be a noetherian fs log scheme,
let $\Lambda$ be a torsion commutative ring killed by an integer invertible on $S$,
and let $\cF\in \rD_\letale(S,\Lambda)$.
Then for every $X\in \lSm_{\ul{S}}$,
we have
\[
R\Gamma_\letale(X\times_{\ul{S}} S,\cF\vert_{X\times_{\ul{S}}S})
\simeq
R\Gamma_\letale((X-\partial X)\times_{\ul{S}} S,\cF\vert_{(X-\partial X)\times_{\ul{S}}S}).
\]
\end{lem}
\begin{proof}
The presheaf $\cF$ given by
\[
X\in \lSm_{\ul{S}}
\mapsto
R\Gamma_\letale(X\times_{\ul{S}} S,\cF\vert_{X\times_{\ul{S}}S})
\]
satisfies $\A^1$-invariance by Theorem \ref{A1.1}, $\square$-invariance by Theorem \ref{A1.2}, and log \'etale descent.
Hence $\cF$ belongs to the full subcategory $\cC$ of $\rD(\Sh_{\letale}(\lSm_S,\Lambda))$ spanned by $\A^1$-invariant and $\square$-invariant objects.
Using \cite[Proposition 2.24]{logA1} to the obvious functor $\lSm_S\to \cC$,
we see that $\cF$ is invariant under verticalization,
which is what we want.
\end{proof}

\section{Relative log \'etale rigidity}
\label{rig}

In this section,
we prove relative log \'etale rigidity.
We begin with a result using $\A^1$-invariance and invariance under verticalization of log \'etale cohomology.

\begin{lem}
\label{rig.7}
Let $S$ be a noetherian fs log scheme,
and let $\Lambda$ be a torsion commutative ring killed by an integer invertible on $S$.
Then for every $\cF\in \rD_\letale(S,\Lambda)$,
$\rho_{S\sharp} \cF\in \rD(\Sh_\letale(\lSm_S,\Lambda))$ is $\A^1$-invariant and invariant under verticalization.
\end{lem}
\begin{proof}
We argue as in \cite[Sous-lemme 4.7]{Ayo14}.
Let $f\colon X\to S$ be a log smooth morphism of noetherian fs log schemes,
and let $j\colon U:=X-\partial_S X\to X$ be the obvious open immersion.
We only need to show
\[
\hom_{\rD(\Sh_\letale(\lSm_S,\Lambda))}(\Lambda_V,f_*f^*\rho_{S\sharp} \cF)
\simeq
\hom_{\rD(\Sh_\letale(\lSm_S,\Lambda))}(\Lambda_V,f_*j_*j^*f^*\rho_{S\sharp} \cF)
\]
for every $V\in \lSm_S$.
By adjunction,
we reduce to the case where $V=S$.
Since $\rho_{(-)\sharp}$ commutes with $f^*$ and $j^*f^*$,
it suffices to show
\[
\hom_{\rD(\Sh_\letale(\lSm_X,\Lambda))}(\Lambda_X,\rho_{X\sharp} f^*\cF)
\simeq
\hom_{\rD(\Sh_\letale(\lSm_{U},\Lambda))}(\Lambda_{U},\rho_{U\sharp} j^*f^*\cF).
\]
By Proposition \ref{DA.2},
$\rho_{X\sharp}$ and $\rho_{U\sharp}$ are fully faithful.
Hence it suffices to show
\[
\hom_{\rD_\letale(X,\Lambda)}(\Lambda_X,f^*\cF)
\simeq
\hom_{\rD_\letale(U,\Lambda)}(\Lambda_{U}, j^*f^*\cF).
\]
This follows from Theorem \ref{ver.1}.

The proof of $\A^1$-invariance is similar using Theorem \ref{A1.1}.
\end{proof}

Let us argue as Ayoub \cite[Corollaire 4.11]{Ayo14} for the fully faithful part of relative log \'etale rigidity.

\begin{prop}
\label{rig.2}
Let $S$ be a noetherian fs log scheme,
and let $\Lambda$ be a torsion commutative ring killed by an integer invertible on $S$.
Then the functor
\[
\rho_{!}\colon \rD_\letale(S,\Lambda)\to \DA_\letale(S,\Lambda)
\]
in \eqref{DA.0.1}
is fully faithful.
\end{prop}
\begin{proof}
Consider the composite functor
\[
\rho_{\sharp}
\colon
\rD_\letale(S,\Lambda)
\xrightarrow{\rho_{S\sharp}}
\rD(\Sh_\letale(\lSm_S,\Lambda))
\xrightarrow{L_\mot}
\DA_\letale^\eff(S,\Lambda).
\]
We need to show
\[
\rho_\sharp \cF
\simeq
\colim_n \ul{\Hom}(T^{\otimes n},T^{\otimes n} \otimes \rho_\sharp\cF)
\]
for every $\cF\in \rD_\letale(S,\Lambda)$,
where $\ul{\Hom}$ denotes the internal Hom.
By Proposition \ref{dim.10},
$\ul{\Hom}(T,-)$ preserves colimits.
Hence it suffices to show
\[
\rho_\sharp \cF
\simeq
\colim_n \ul{\Hom}(T^{\otimes n} ,\colim_m \ul{\Hom}(T^{\otimes m},T^{\otimes m+n} \otimes \rho_\sharp\cF)).
\]
By \cite[Proposition 4.10]{Ayo14},
we have a natural morphism
\[
\alpha_{\ul{S}}^n
\colon
T^{\otimes n} \to \rho_\sharp(\Lambda_{\ul{S}}(n))[2n]
\]
in $\DA_\et^\eff(\ul{S},\Lambda)$ such that $\Sigma^\infty \alpha_{\ul{S}}^n$ is an isomorphism.
Pulling back $\alpha_{\ul{S}}^n$ to $S$,
we obtain a natural morphism
\[
\alpha_{S}^n
\colon
T^{\otimes n} \to \rho_\sharp(\Lambda_{S}(n))[2n]
\]
in $\DA_\letale^\eff(S,\Lambda)$ such that $\Sigma^\infty \alpha_{S}^n$ is an isomorphism.
It follows that we have
\[
\colim_m \ul{\Hom}(T^{\otimes m},T^{\otimes m+n}\otimes \rho_\sharp \cF)
\simeq
\colim_m \ul{\Hom}(T^{\otimes m},T^{\otimes m}\otimes \rho_\sharp (\cF(n))[2n]).
\]
Hence it suffices to show
\[
\rho_\sharp \cF
\simeq
\colim_n \ul{\Hom}(T^{\otimes n} , \colim_m \ul{\Hom}(T^{\otimes m},T^{\otimes m}\otimes \rho_\sharp (\cF(n))[2n])).
\]
Since $\ul{\Hom}(T,-)$ preserves colimits.
it suffices to show
\[
\rho_\sharp \cF
\simeq
\colim_m \ul{\Hom}(T^{\otimes m} , \colim_n \ul{\Hom}(T^{\otimes n},T^{\otimes m}\otimes \rho_\sharp (\cF(n))[2n])).
\]
Since $\rho_\sharp (\cF(n))\simeq \rho_\sharp(\Lambda_S(n)) \otimes \rho_\sharp \cF$,
$\alpha_S^n$ induces a morphism
\[
T^{\otimes n} \otimes \rho_\sharp \cF
\to
\rho_\sharp (\cF(n))[2n]
\]
that becomes an isomorphism after applying $\Sigma^\infty$.
We have 
\[
\rho_{S\sharp} (\cF(n))[2n]
\simeq
\ul{\Hom}(T,\rho_{S\sharp} (\cF(n+1))[2n+2])
\]
in $\rD(\Sh_\letale(\lSm_S,\Lambda))$,
so we have
\[
\rho_{\sharp} (\cF(n))[2n]
\simeq
\ul{\Hom}(T,\rho_\sharp (\cF(n+1))[2n+2])
\]
in $\DA_\letale^\eff(S,\Lambda)$ by Lemma \ref{rig.7}.
Hence we can define
\[
\bF
:=
(\rho_\sharp \cF,\rho_\sharp (\cF(1))[2],\rho_\sharp (\cF(2))[4],\ldots)\in \DA_\letale(S,\Lambda).
\]
The isomorphism
\[
\Sigma^\infty \alpha_S^m \otimes \id
\colon
T^{\otimes m}
\otimes \bF
\to
\rho_\sharp(\Lambda_S(m))[2m] \otimes \bF
\]
yields an isomorphism
\begin{align*}
& \colim_n \ul{\Hom}(T^{\otimes n},T^{\otimes m} \otimes \rho_\sharp(\cF(n))[2n])
\\
\simeq &
\colim_n \ul{\Hom}(T^{\otimes n}, \rho_\sharp(\cF(m+n))[2m+2n]).
\end{align*}
Hence it suffices to show
\[
\rho_\sharp \cF
\simeq
\colim_m \ul{\Hom}(T^{\otimes m}, \colim_n \ul{\Hom}(T^{\otimes n}, \rho_\sharp(\cF(m+n))[2m+2n])).
\]
This is equivalent to
\[
\rho_\sharp \cF
\simeq
\colim_n \ul{\Hom}(T^{\otimes n}, \rho_\sharp(\cF(n))[2n])
\]
since $\ul{\Hom}(T,-)$ preserves colimits,
which is clear.
\end{proof}

\begin{cor}
\label{rig.5}
Let $S$ be a noetherian fs log scheme,
and let $\Lambda$ be a torsion commutative ring killed by an integer invertible on $S$.
Then we have an induced fully faithful functor
\[
\rho_!
\colon
\rD_\ket(S,\Lambda)
\to
\DA_\letale^\ex(S,\Lambda).
\]
\end{cor}
\begin{proof}
The composite
\[
\rD_\ket(S,\Lambda)
\xrightarrow{\kappa^*}
\rD_\letale(S,\Lambda)
\xrightarrow{\rho_!}
\DA_\letale(S,\Lambda)
\]
sends $\Lambda_X$ to $\Lambda_X$ for every $X\in S_\ket$,
so the essential image lies in $\DA_\letale^\ex(S,\Lambda)$ since $X$ is exact over $S$.
Propositions \ref{dim.5} and \ref{rig.2} finish the proof.
\end{proof}

\begin{prop}
\label{rig.8}
Let $S$ be a noetherian fs log scheme,
and let $\Lambda$ be a torsion commutative ring killed by an integer invertible on $S$.
Then $\rho_!\colon \rD_\ket(S,\Lambda)\to \DA_\letale^\ex(S,\Lambda)$ commutes with $(n)$ for every integer $n$.
The analogous statement holds in the log \'etale case too.
\end{prop}
\begin{proof}
We focus on the Kummer \'etale case since the proofs are similar.

As noted in the proof of Proposition \ref{rig.2},
we have
\[
\Sigma^\infty \alpha_S^n\colon \Lambda_S(n)\xrightarrow{\simeq} \rho_! (\Lambda_S(n))
\]
for every integer $n\geq 0$.
Since $\rho$ is symmetric monoidal,
we see that $\rho_!$ commutes with $(n)$ for every $n\geq 0$.
For every $\cF\in \rD_\ket(S,\Lambda)$,
we have
\[
\rho_! (\cF(-n)) (n) \simeq \rho_! (\cF(-n) (n)) \simeq \rho_! \cF.
\]
Hence $\rho_!$ commutes with $(-n)$ as well.
\end{proof}

For a finite dimensional noetherian fs log scheme $S$,
let $\lSch_S^\Zar$ be the category of fs log schemes of finite type over $S$ admitting charts Zariski locally.

\begin{rmk}
Recall that $B$ is a finite dimensional noetherian base scheme.

According to \cite[Theorem 1.1.1]{logshriek},
$\SH$ is a log motivic $\infty$-category in the sense of \cite[Definition 2.1.1]{logshriek} on the category $\lSch^\Zar_B$.
A similar proof shows that $\DA_\letale(-,\Lambda)$ is a log motivic $\infty$-category on $\lSch^\Zar_B$.
This establishes \cite[Theorems 1.2.1, 1.3.1]{logsix} for $\DA_\letale^\ex(-,\Lambda)$ on $\lSch^\Zar_B$.
\end{rmk}

Let us argue as Cisinski--D\'eglise \cite[Proposition 4.4.3, Theorem 4.5.2]{CDetale} for the essentially surjective part of relative log \'etale rigidity.

\begin{prop}
\label{rig.3}
Let $f\colon X\to S$ be a proper morphism of noetherian fs log schemes,
and let $\Lambda$ be a torsion commutative ring killed by an integer invertible on $S$.
Then
\[
\rho_! f_*
\simeq
f_* \rho_! \colon
\rD_\ket(X,\Lambda)
\to
\DA_\letale^\ex(S,\Lambda).
\]
\end{prop}
\begin{proof}
It suffices to show
\[
\hom_{\DA_\letale^\ex(S,\Lambda)}(\Lambda_U(n),\rho_! f_* \cF)
\simeq
\hom_{\DA_\letale^\ex(S,\Lambda)}(\Lambda_U(n),f_*\rho_! \cF)
\]
for every exact log smooth morphism $g\colon U\to S$, $\cF\in \rD_\ket(X,\Lambda)$, and integer $n$.
Since $f^*$ commutes with $(n)$, by taking right adjoints, we see that $f_*$ commutes with $(-n)$ and hence with $(n)$.
Also, $\rho_!$ commutes with $(n)$ by Proposition \ref{rig.8}.
Hence we reduce to the case where $n=0$.

Note that $g^*$ commutes with $f_*$ by Theorem \ref{dim.4}(2).
Also, $g^*$ commutes with $\rho_!$.
By adjunction,
replacing $S$ by $U$,
we reduce to the case where $U=S$.
Since $\Lambda_S\simeq \rho_! \Lambda_S$ and $\rho_!$ is fully faithful by Corollary \ref{rig.5},
it suffices to show
\[
\hom_{\rD_\ket(S,\Lambda)}(\Lambda_S,f_* \cF)
\simeq
\hom_{\DA_\letale^\ex(S,\Lambda)}(\Lambda_S,f_*\rho_! \cF).
\]
This is equivalent to
\[
\hom_{\rD_\ket(X,\Lambda)}(\Lambda_X,\cF)
\simeq
\hom_{\DA_\letale^\ex(X,\Lambda)}(\Lambda_X,\rho_! \cF).
\]
This holds since $\rho_!$ is fully faithful and $\rho_! \Lambda_X \simeq \Lambda_X$.
\end{proof}

Now, we finish the proof of relative Kummer \'etale rigidity.

\begin{thm}
\label{rig.4}
Let $S$ be a finite dimensional noetherian fs log scheme,
and let $\Lambda$ be a torsion commutative ring killed by an integer invertible on $S$.
Then the functor
\[
\rho_!
\colon
\rD_\ket(S,\Lambda)
\to
\DA_\letale^\ex(S,\Lambda)
\]
is an equivalence of symmetric monoidal $\infty$-categories.
\end{thm}
\begin{proof}
The claim is strict \'etale local on $S$.
Hence we may assume that $S$ has charts Zariski locally.

By Corollary \ref{rig.5},
it remains to show that $\rho_!$ is essentially surjective.
Let $X\to S$ be a vertical exact log smooth morphism of fs log schemes.
By Proposition \ref{rig.8},
we only need to show that $\Lambda_X\in \DA_\letale^\ex(S,\Lambda)$ is in the essential image of $\rho_!$ strict \'etale locally on $X$.
We may assume that the tangent bundle $T_f$ of $f$ is trivial, $X$ has a chart, and $f$ is separated and has pure relative dimension $d$.

Consider the induced factorization
\[
X\xrightarrow{g} \ul{X}\times_{\ul{S}} S\xrightarrow{j} Y\xrightarrow{p} S
\]
of $f$,
where $Y$ is obtained by the Nagata compactification for $\ul{X}\times_{\ul{S}}S \to S$, and $g$ is an induced proper morphism.
Then $f_!\simeq p_*j_!g_*$.
By \cite[Theorem 1.3.1(2)]{logsix},
we have
\[
f_\sharp \simeq f_!(d)[2d]
\colon
\DA_\letale^\ex(X,\Lambda)
\to
\DA_\letale^\ex(S,\Lambda)
\]
since $T_f$ is trivial.
Propositions \ref{rig.8} and \ref{rig.3} yield
\[
\rho_!(p_*j_!g_*\Lambda_X(d)[2d])
\simeq
f_!\Lambda_X (d)[2d].
\]
Combine what we have discussed above to conclude.
\end{proof}

We also prove relative log \'etale rigidity as follows.

\begin{thm}
\label{rig.6}
Let $S$ be a finite dimensional noetherian fs log scheme,
and let $\Lambda$ be a torsion commutative ring killed by an integer invertible on $S$.
Then the functor
\[
\rho_!
\colon
\rD_\letale(S,\Lambda)
\to
\DA_\letale(S,\Lambda)
\]
is an equivalence of symmetric monoidal $\infty$-categories.
\end{thm}
\begin{proof}
Let $f\colon X\to S$ be a log smooth morphism of fs log schemes,
and let $n$ be an integer.
We need to show that $f_\sharp \Lambda_X(n)$ is contained in the essential image of $\rho_!$.
Since $\rho_!$ commutes with $g_\sharp$ and $g^*$ for log \'etale morphisms $g$,
the question is log \'etale local on $S$.
Hence by \cite[Theorem III.2.6.7]{Ogu},
we may assume that $f$ is exact log smooth.
Theorem \ref{rig.4} finishes the proof.
\end{proof}

In particular,
under the above notation,
we have the realization functors
\[
\SH(S)\to \rD_\letale(S,\Lambda),
\;
\SH^\ex(S)\to \rD_\ket(S,\Lambda).
\]
Let us construct more refined realization functors as follows.

\begin{const}
Let $S$ be a finite dimensional noetherian fs log scheme over a perfect field $k$,
and let $\Lambda$ be a torsion commutative ring killed by an integer invertible in $k$.
Consider the adjoint functors
\[
\gamma^* : \DA(k,\Lambda) \rightleftarrows  \DM(k,\Lambda) 
:\gamma_*
\]
in \cite[11.2.16]{CD19}.
Then $\gamma_* \Lambda_k$ is the motivic cohomology spectrum $\rM \Lambda_k$.

We also have the \'etale version
\[
\gamma_\et^* : \DA_\et(k,\Lambda) \rightleftarrows \DM_\et(k,\Lambda) 
:\gamma^\et_*.
\]
The \'etale hyperlocalization functor $L_\et \colon \DA(k,\Lambda)\to \DA_\et(k,\Lambda)$ sends $\rM \Lambda_k$ to $\gamma^{\et}_* \Lambda_k$ since the \'etale hypersheafification of the motivic cohomology complex is indeed the \'etale motivic cohomology complex as a consequence of \cite[Proposition 10.7]{MVW}.
We know that $\gamma^*_\et$ is an equivalence by \cite[Theorems 7.20, 9.35]{MVW} and \cite[Th\'eor\`eme 4.1]{Ayo14},
so we have
\begin{equation}
\label{rig.9.1}
\Lambda_k
\simeq
L_\et \rM \Lambda_k
\end{equation}
in $\CAlg(\DA_\et(k,\Lambda))$.
Let $p\colon S\to \Spec(k)$ be a structural morphism,
and let $\rM\Lambda_S:=p^* \rM \Lambda_k$.
Applying $p^*$ to \eqref{rig.9.1},
we have
\begin{equation}
\label{rig.9.2}
\Lambda_S \simeq L_\letale \rM \Lambda_S
\end{equation}
in $\CAlg(\DA_\letale(S,\Lambda))$,
where $L_\letale$ is the hyperlocalization functor.

Now, the \emph{log \'etale realization functor}
\[
\mathrm{Re}_\letale\colon \Mod_{\rM \Lambda_S}(\SH(S))
\to
\rD_\letale(S,\Lambda)
\]
is defined to be the composite
\begin{align*}
&\Mod_{\rM \Lambda_{S}}(\SH(S))
\to
\Mod_{\rM \Lambda_{S}}(\DA(S,\Lambda))
\\
\xrightarrow{L_\letale} &
\Mod_{L_\letale \rM \Lambda_{S}}(\DA_\letale(S,\Lambda))
\xrightarrow{\simeq}
\DA_\letale(S,\Lambda)
\xrightarrow{\simeq}
\rD_\letale(S,\Lambda),
\end{align*}
where the equivalences are obtained by Theorem \ref{rig.6} and \eqref{rig.9.2}.
Use Theorem \ref{rig.4} and restrict $\mathrm{Re}_\letale$ to define
the \emph{Kummer \'etale realization functor}
\[
\mathrm{Re}_\ket\colon 
\Mod_{\rM \Lambda_{S}}^\ex(\SH(S))
\to
\rD_\ket(S,\Lambda),
\]
where $\Mod_{\rM \Lambda_{S/B}}^\ex(\SH(S))$ is the full subcategory of $\Mod_{\rM \Lambda_{S/B}}(\SH(S))$ generated under colimits, shifts, and Tate twists by exact log smooth motives.
\end{const}

\section{Proof of Theorem \ref{intro.1}}
\label{proof}

Recall that $B$ is a finite dimensional noetherian base scheme.
By Theorem \ref{rig.4},
it suffices to show (1)--(10) in Theorem \ref{intro.1} for either $\rD_\ket(-,\Lambda)$ or $\DA_\letale^\ex(-,\Lambda)$.

(1)--(3) It is formal that these are satisfied for $\DA_\letale^\ex(-,\Lambda)$.

(4) Let $f\colon X\to S$ be a virtual isomorphism in $\lSch_B$.
We can work \'etale locally on $\ul{S}\simeq \ul{X}$,
so we may assume that $X$ and $S$ have charts.
By \cite[Theorem 3.2.9]{logsix},
(4) holds for $\DA_\letale^\ex(-,\Lambda)$.

(5) See Proposition \ref{dim.7}.

(6) The $1$-categorical construction of $(-)_!$ for $\rD_\ket(-,\Lambda)$ is due to \cite[5.4]{Nak1}.
The argument as in \cite[Construction 3.7.4]{logshriek} enhances this $\infty$-categorically.
Proposition \ref{dim.3} implies that $f_!$ admits a right adjoint for every separated morphism $f$ in $\lSch_B$.

(7) Let $f\colon X\to S$ be a separated morphism in $\lSch_B$,
let $\cF\in \DA_\letale^\ex(X,\Lambda)$,
and let $\cG\in \DA_\letale^\ex(S,\Lambda)$.
There exists a commutative square
\[
\begin{tikzcd}
X_0\ar[r,"f_0"]\ar[d,"q_0"']&
S_0\ar[d,"p_0"]
\\
X\ar[r,"f"]&
S
\end{tikzcd}
\]
such that $X_0$ and $S_0$ have charts Zariski locally and $p_0$ and the induced morphism $X_0\to X\times_S S_0$ are separated strict \'etale coverings.
Let $S_\bullet$ and $X_\bullet$ be the \v{C}ech nerves of $p_0$ and $q_0$,
let $p_\bullet\colon S_\bullet\to S$ and $q_\bullet\colon X_\bullet\to X$ be the projections,
and let $f_n\colon X_n\to S_n$ be the induced morphism.
By strict \'etale descent and the projection formula for strict \'etale morphisms,
we have
\begin{gather*}
f_! \cF\otimes \cG
\simeq
\colim_{n\in \Delta^\op}
f_! q_{n_\sharp} q_n^* \cF \otimes \cG
\simeq
\colim_{n\in \Delta^\op}
p_{n\sharp} f_{n!}q_n^* \cF\otimes \cG
\simeq
\colim_{n\in \Delta^\op}
p_{n\sharp}(f_{n!} q_n^* \cF \otimes p_n^* \cG),
\\
f_!(\cF\otimes f^* \cG)
\simeq
\colim_{n\in \Delta^\op}
f_! q_{n\sharp} q_n^* (\cF\otimes f^* \cG)
\simeq
\colim_{n\in \Delta^\op}
p_{n\sharp} f_{n!} (q_n^* \cF \otimes f_n^* p_n^* \cG).
\end{gather*}
We also have $f_{n!}q_n^* \cF\otimes p_n^*\cG\simeq f_{n!}(q_n^*\cF\otimes f_n^*p_n^*\cG)$ by \cite[Theorem 3.7.8]{logshriek}.
Combine these equations to prove (7) for $\DA_\letale^\ex(-,\Lambda)$.

(8) See Theorem \ref{dim.4}.

(9) Let $f\colon X\to S$ be a separated vertical exact log smooth morphism in $\lSch_B$ with pure relative dimension $d$.
Let us use the notation in the proof of (7).
We have $f_{n\sharp} \simeq f_{n!}(-\otimes \Th(T_{f_n}))$ by \cite[Theorem 1.3.1(2)]{logsix},
where $T_{f_n}$ is the tangent bundle of $f_n$.
Since $\ul{T_{f_n}}\to \ul{X_n}$ is a rank $d$ vector bundle,
we have a natural isomorphism
\[
\Th(\ul{T_{f_n}})\simeq \Lambda_{\ul{X_n}} (d)[2d]
\]
in $\DA_\et(\ul{X_n},\Lambda)$
using the orientation theory \cite[Definition 4.1.4]{CDetale}, \cite[Example 2.4.40]{CD19}.
By pulling back to $X_n$,
we have a natural isomorphism
\[
\Th(T_{f_n})\simeq \Lambda_{X_n} (d)[2d]
\]
in $\DA_\letale^\ex(X_n,\Lambda)$.
Hence we have
\(
f_{n\sharp} \simeq f_{n!}(d)[2d].
\)
By strict \'etale descent,
we have
\begin{gather*}
f_\sharp
\simeq
\colim_{n\in \Delta^\op} f_\sharp q_{n\sharp}q_n^*
\simeq
\colim_{n\in \Delta^\op}  p_{n\sharp} f_{n\sharp} q_n^*,
\;
f_!
\simeq
\colim_{n\in \Delta^\op} f_! q_{n\sharp}q_n^*
\simeq
\colim_{n\in \Delta^\op}  p_{n\sharp} f_{n!} q_n^*.
\end{gather*}
Combine these equations to prove (9) for $\DA_\letale^\ex(-,\Lambda)$.

(10) Argue as in (9) to reduce to \cite[Theorem 1.3.1(3)]{logsix}.

This completes the proof of Theorem \ref{intro.1}. \hfill \qedsymbol

\bibliography{bib}
\bibliographystyle{siam}

\end{document}